\documentclass{article}
\usepackage{fullpage,diagbox}
\usepackage{mathtools, appendix}
\mathtoolsset{showonlyrefs}

\usepackage[ruled,linesnumbered, vlined, noend]{algorithm2e}

\usepackage{amsthm, amsfonts, amssymb, amsmath,enumitem, natbib, bbm, mathabx}
\newtheorem{thm}{Theorem}
\newtheorem{lem}{Lemma}
\newtheorem{cor}{Corollary}
\newtheorem{prop}{Proposition}

\newcounter{myalgctr}
\newenvironment{rem}{
   \vskip1mm\indent
   \refstepcounter{myalgctr}
   \textbf{Remark \themyalgctr}
   }{\hfill$\diamond$\par}  
\RequirePackage[colorlinks=true, pdfstartview=FitV,
linkcolor=blue, citecolor=blue, urlcolor=blue,hypertexnames=false]{hyperref}
\numberwithin{myalgctr}{section}
\providecommand{\norm}[1]{\left\lVert#1\right\rVert}
\DeclareMathOperator*{\argmin}{arg\,min}

\makeatletter
\def\namedlabel#1#2{\begingroup
    #2%
    \def\@currentlabel{#2}%
    \phantomsection\label{#1}\endgroup
}
\makeatother
\newcommand{\vertiii}[1]{{\left\vert\kern-0.25ex\left\vert\kern-0.25ex\left\vert #1
    \right\vert\kern-0.25ex\right\vert\kern-0.25ex\right\vert}}

\newcommand{\RIP}{\mbox{RIP}}
\usepackage{color, float,authblk}
\title{\bf Uniform-in-Submodel Bounds for Linear Regression in a Model Free Framework}
\author[1]{Arun Kumar Kuchibhotla
    \thanks{Email: {\tt arunku@cmu.edu}}}
\author[2]{Lawrence D. Brown}
\author[2]{Andreas Buja}
\author[2]{Edward I. George}
\author[2]{Linda H. Zhao}
\affil[1]{Department of Statistics \& Data Science, Carnegie Mellon University}
\affil[2]{Department of Statistics, The Wharton School, University of Pennsylvania.}
\date{}
\begin{document}

\maketitle







\begin{abstract}
  For the last two decades, high-dimensional data and methods have
  proliferated throughout the literature.  Yet, the classical
  technique of linear regression has not lost its usefulness in
  applications.  In fact, many high-dimensional estimation techniques
  can be seen as variable selection that leads to a smaller set of
  variables (a ``sub-model'') where classical linear regression
  applies.  We analyze linear regression estimators resulting from
  model-selection by proving estimation error and linear
  representation bounds uniformly over sets of submodels.  Based on
  deterministic inequalities, our results provide ``good'' rates when
  applied to both independent and dependent data.  These results are
  useful in meaningfully interpreting the linear regression estimator
  obtained after exploring and reducing the variables and also in
  justifying post model-selection inference.  All results are derived
  under no model assumptions and are non-asymptotic in nature.\let\thefootnote\relax\footnotetext{The authors thank Abhishek Chakrabortty for discussions that led to Remark~\ref{rem:GramCov}. {We also thank the reviewers and the editor for their constructive comments which have led to a better presentation.}}
\end{abstract}
\section{Introduction and Motivation}
Least squares linear regression is one of the most widely used prediction tools in practical data analysis. With its simple form, linear regression leads to interpretable results and in many cases has predictive performance on par with sophisticated/complex models. It is, however, an open secret that in most cases the set of covariates used in the final linear regression model is rarely the same as the set of covariates initially considered by the data analyst. This is typically a consequence of the selection of a good predictive submodel based on an estimate of the out-of-sample prediction risk. We use  ``submodel'' here to denote a subset of the full set of covariates.

Traditional analysis of the least squares linear regression estimator restricts attention to a single set of covariates to prove consistency as well as asymptotic normality; see~\cite{white1980heteroskedasticity,white1980using} and~\cite{Buja14}. In this case, it was proved that the least squares estimator is weakly and strongly consistent to the population least squares functional; see~\eqref{eq:LeastSquaresExpected} below. Also, a properly normalized estimator has an asymptotic normal distribution. However, the theoretical understanding and practical usefulness of submodel least squares estimators resulting from a covariate selection procedure requires simultaneous consistency and (asymptotic) normality of all the estimators under consideration. Such simultaneous consistency and normality properties are the major focus of the current article. These are what we call \emph{uniform-in-submodel} results. To be more concrete, suppose $\mathcal{M} = \{M_1, M_2, \ldots, M_L\}$ denotes a collection of submodels, where $M_j$ represents a subset of covariates for $1\le j\le L$. Also, let $\hat{\beta}_{M_j}$ represent the least squares estimator for the linear regression of the response on the covariates in $M_j$. By simultaneous consistency, we mean the existence of target vectors $\{\beta_{M_j}:1\le j\le L\}$ such that
\begin{equation}\label{eq:SimulConsistency}
\sup_{M\in\mathcal{M}}\,\|\hat{\beta}_{M} - \beta_M\| ~=~ o_p(1),\quad\mbox{as}\quad n\to\infty,
\end{equation}
for some norm $\norm{\cdot}$. To claim simultaneous asymptotic normality, we prove the existence of functions $\{\psi_{M_j}(\cdot):\,1\le j\le L\}$ such that
\begin{equation}\label{eq:SimulNormality}
\sup_{M\in\mathcal{M}}\,\norm{\sqrt{n}\left(\hat{\beta}_M - \beta_M\right) - \frac{1}{\sqrt{n}}\sum_{i=1}^n \psi_M(Z_i)} ~=~ o_p(1),\quad\mbox{as}\quad n\to\infty.
\end{equation}
Here $n$ represents the sample size and $Z_i = (X_i, Y_i), 1\le i\le n$ represent the regression data, with detailed notation given in Section~\ref{sec:NotationSetting}. Equation~\eqref{eq:SimulNormality} provides the well-known ``asymptotic uniform linear representation'' in the special case of the least squares linear regression estimator. {This uniform linear representation is very crucial in providing inference after variable selection via simultaneous inference~\citep{Bac16}. If $\widehat{M}$ is a selected model, then one can perform inference on $\beta_{\widehat{M}}$ by estimating the distribution of~$\widehat{\beta}_{\widehat{M}}$. This can be a tricky problem to deal with as shown in the works of~\cite{leeb2005model,Leeb06,leeb2006performance,leeb2008can}.}

Although various model selection criteria like $C_p$, AIC, BIC, and lasso have been recommended for covariate selection in linear regression, results of the type~\eqref{eq:SimulConsistency} and~\eqref{eq:SimulNormality} have not been established in the literature (at least not in the full generality considered here). Our method of attack is quite non-standard. Instead of assuming that the observations are independent and identically distributed, we prove a purely  deterministic inequality to bound the left hand sides of~\eqref{eq:SimulConsistency} and~\eqref{eq:SimulNormality} using maxima of several averages.  We then control these averages under both independence and functional dependence to obtain explicit rates of convergence; cf.~\cite{White2001} where a detailed classical analysis of the least squares regression estimator is provided. The functional dependence structure of data, introduced in \cite{Wu05}, is based on the idea of coupling and covers the setting of many linear and non-linear time series. This dependence concept is very closely related to the $L_p$-approximability concept introduced in~\cite{Potscher97}.

Some noteworthy aspects of our results are as follows.
\begin{enumerate}
  \item We provide a purely deterministic inequality for the least squares linear regression estimator which does not require any stochasticity of the regression data and holds for any sample size $n$. These deterministic results are sharp and by nature more widely applicable than any asymptotic results. Some deterministic inequalities for linear regression appeared in~\cite{kuchibhotla2018valid}.  Although these inequalities led to suboptimal rates,  the structure of those deterministic inequalities were useful for the context in that paper. 
  \item All our results allow misspecification of the linear model. This means that the classical Gauss-Markov linear model need not hold true for any of the submodels under consideration; see Chapter 4 of~\cite{monahan2008primer}. Two important objections (for us) to the classical model are the impositions of fixed design and linearity structure on the data generating distribution. Since our setting allows for misspecification, we call our framework ``model-free.'' We note here that our results do apply to the setting of fixed covariates.
  \item When studied assuming a suitable randomness structure (such as independence or functional dependence), our results are precise concentration inequalities applicable in finite samples and apply to high-dimensional observations. Another interesting facet of our results is that we do not assume the observations are identically distributed.  This is an important generalization {needed to include the case of fixed covariates}. 
  \item For concreteness, we take the set of submodels $\mathcal{M}$ to be the set of all submodels of size bounded by $k$ (for some $1\le k\le p$). Here $p$ represents the total number of available covariates. Under certain regularity conditions, the rates of convergence we obtain in this case for simultaneous consistency~\eqref{eq:SimulConsistency} and normality~\eqref{eq:SimulNormality} with Euclidean norm are $\sqrt{k\log(ep/k)/n}$ and $k\log(ep/k)/\sqrt{n}$, respectively (up to a lower order additive term). Interestingly, the simultaneous consistency rate matches the minimax optimal rate of a well-specified high-dimensional sparse linear regression; see~\cite{Raskutti11}. It should be noted that even though the rates match with the setting of \emph{well-specified high-dimensional linear regression}, we do NOT require a well-specified model in this article. 
  \item In the process of applying our results to functionally dependent observations, we prove a tail bound for zero mean dependent sums, thereby extending the results of \cite{WuWu16}. For independent observations, we use the precise concentration inequality results of~\cite{KuchAbhi17}.
\end{enumerate}

In addition to the important general model selection problem above where the results of the type~\eqref{eq:SimulConsistency} and~\eqref{eq:SimulNormality} are required, our simultanety results can be seen to provide essential inferential validity guarantees for the following setting of growing importance. 
In the vast literature on high-dimensional linear regression, it has become customary to assume an underlying linear model along with a sparsity constraint on the true regression parameter.  But suppose statisticians are not willing to assume sparsity of the parameter, and neither are they willing to assume a linear model. Such unwillingness is not unreasonable in light of the fact that any model is just an approximation, and sparsity is just an assumption of convenience. Now consider the following stylized description of approaches to high-dimensional data as widely practiced in applied statistics and data science: High dimensional data is first explored either in a formal algorithmic way (e.g., using lasso or best subset selection) and/or in an informal exploratory way (e.g., using residual and leverage plots) to select a manageable small set of variables. Subsequently, the reduced data is subjected to linear regression. The combination of variable selection and linear regression is thought of as one procedure, a ``high-dimensional linear regression''.  Even though the procedure uses only a reduced set of variables in the final regression, {it} uses all the variables in the preceding selection phase. Suppose $\hat{M}\in\mathcal{M}$ is the final selected submodel (from some collection of models $\mathcal{M}$) and $\hat{\beta}_{\hat{M}}$ is the least squares linear regression estimator thus obtained. The estimator $\hat{\beta}_{\hat{M}}$ is known as the post-regularization estimator in the high-dimensional statistics literature if $\hat{M}$ is obtained from some regularized least squares procedure. An important question now is ``what does $\hat{\beta}_{\hat{M}}$ estimate (consistently)?''. A simultaneous result answers this question through the \emph{trivial} bound
\[
\|\hat{\beta}_{\hat{M}} - \beta_{\hat{M}}\| ~\le~ \sup_{M\in\mathcal{M}}\,\|\hat{\beta}_M - \beta_M\| = o_p(1). 
\]
Therefore, $\hat{\beta}_{\hat{M}}$ is estimating the quantity $\beta_{\hat{M}}$ which is random through $\hat{M}$. If the model selection procedure is such that $\hat{M}$ does not stabilize as $n\to\infty$, then $\hat{\beta}_{\hat{M}}$ is only consistent for the random quantity $\beta_{\hat{M}}$ and may not be consistent for any non-random quantity. By comparison if $\mathbb{P}(\hat{M} = M_0) \to 1$ as $n\to\infty$ for some submodel $M_0$, then with probability converging to one $\beta_{\hat{M}} = \beta_{M_0}$ and hence $\hat{\beta}_{\hat{M}}$ is consistent for the non-random quantity $\beta_{M_0}$.
\subsection{Literature Review}
Results of the simultaneous type described in~\eqref{eq:SimulConsistency} and~\eqref{eq:SimulNormality} are not readily available in the literature. Some works that are closely related to ours are~\cite{Belloni13},~\cite{bachoc2018post} and~\cite{chakrabortty18}. Although some of these works consider a simultaneous problem, their results are only restricted to certain special cases (e.g., independent observations and/or fixed design) of our framework. \cite{Belloni13} prove the rate of convergence of the least squares linear regression estimator obtained after covariate selection using lasso. \cite{bachoc2018post} prove the rate of convergence of 
\[
\sup_{M\in\mathcal{M}}\,\norm{\hat{\beta}_M - \beta_M}_{\infty}
\]
under the restricted isometry property (RIP). (Here $\norm{v}_{\infty}$ for a vector $v$ denotes the maximum absolute entry in the vector.) Also, they only consider fixed covariates. We do not assume RIP because it is not a practical assumption, and also we prove the simultaneous convergence guarantee with the Euclidean norm rather than $\norm{\cdot}_{\infty}$. It should also be mentioned that~\cite{bachoc2018post} appeared after the initial version of the current work~\cite{2018arXiv180205801K}. \cite{chakrabortty18} independently prove results very similar to ours in the case of independent observations with sub-Gaussian tails. They consider a more general collection of submodels $\mathcal{M}$ than the set of $k$-sparse submodels; see section 5 of~\cite{chakrabortty18} for more details. Because our results are deterministic in nature, they do apply for a general collection of submodels, but for concreteness we fix the choice of the collection. {Under the assumptions of~\citet[Section 5]{chakrabortty18}, their results match ours exactly. We note, however, that their results are only proved for independent and identically distributed observations, which is why they do not apply to the case of fixed covariates. Further, our results, including the case of independent observations, are proved for a large class of tail assumptions that subsume their assumptions.

Finally, we mention two recent works that discuss uniform-in-submodel type results. \cite{Rin16} in their Theorem 1, as well as Remark 4 that follows, discuss uniform-in-submodel consistency for independent and identically distributed observations that are bounded. Their rates, however, are suboptimal; for instance, their Theorem 1 only proves a rate $k\sqrt{\log(k)/n}$ while our results imply the optimal rate of $\sqrt{k/n}$. \citet[Chapter 2]{giessing2018high}, following the initial version of our work, proves uniform-in-submodel consistency as well as linear representation results for quantile regression when the observations are independent. The tail assumptions on the observations there are weaker than ours but this is expected, at least for the response, because the loss is Lipschitz in the response.}  

\subsection{Organization}
The remainder of our paper is organized as follows. In Section~\ref{sec:NotationSetting}, we introduce our notation and general framework. In Section~\ref{sec:Deterministic}, we derive various deterministic inequalities for linear regression that form the core of the paper. The application of these results to the case of independent observations is considered in Section~\ref{sec:Independent}. The application of the deterministic inequalities to the case of (functionally) dependent observations is considered in Section~\ref{sec:Dependent}. 
A discussion of our results along with their implications for post-selection inference is given in Section~\ref{sec:Discussion}. Some auxiliary probability results for sums of independent and functionally dependent random variables are given in Appendix~\ref{AppSec:Independence} and Appendix~\ref{AppSec:Dependence}, respectively. 

\section{Notation}\label{sec:NotationSetting}
Suppose $(X_1, Y_1), \ldots, (X_n, Y_n)$ are $n$ random vectors in $\mathbb{R}^p\times\mathbb{R}$. Throughout the paper, we implicitly think of $p$ as a function of $n$ and so the sequence of random vectors should be thought of as a triangular array. The term ``submodel'' is used to specify the subset of covariates used in the regression and does not refer to any probability model. We do \underline{\emph{not}} assume a linear model (in any sense) to be true anywhere for any choice of covariates in any section of the paper. In this sense all our results are applicable in the case of misspecified linear regression models.

For any vector $v\in\mathbb{R}^q$ for $q\ge 1$ and $1\le j\le q$, let $v(j)$ denote the $j$-th coordinate of $v$. For any non-empty submodel $M$ given by a subset of $\{1,2,\ldots,q\}$, let $v(M)$ denote a sub-vector of $v$ with indices in $M$. For instance, if $M = \{2, 4\}$ and $q \ge 4$, then $v(M) = (v(2), v(4))$. The notation $|M|$ is used to denote the cardinality of $M$.  For any non-empty submodel $M\subseteq\{1,2,\ldots,q\}$ and any symmetric matrix $A\in\mathbb{R}^{q\times q}$, let $A(M)$ denote the sub-matrix of $A$ with indices in $M\times M$. For $1\le j, k\le q$, let $A(j,k)$ denote the value at the $j$-th row and the $k$-th column of $A$. Define the $r$-norm of a vector $v\in\mathbb{R}^q$ for $1\le r\le \infty$ as
\[
\norm{v}_r^r := \sum_{j=1}^q |v(j)|^r,\quad\mbox{for}\quad 1\le r < \infty,\quad\mbox{and}\quad \norm{v}_{\infty} := \max_{1\le j\le q}|v(j)|.
\]
Let $\norm{v}_0$ denote the number of non-zero entries in $v$ (note this is not a norm). For any square matrix $A$, let $\lambda_{\min}(A)$ denote the minimum eigenvalue of $A$. Also, let the elementwise maximum and the operator norm be defined, respectively, as
\[
\vertiii{A}_{\infty} := \max_{1\le j, k\le q}|A(j,k)|,\quad\mbox{and}\quad \norm{A}_{op} := \sup_{\norm{\delta}_2 \le 1}\norm{A\delta}_2.
\]
The following simple inequalities are useful. For any matrix $A\in\mathbb{R}^{q\times q}$ and $v\in\mathbb{R}^q$, 
\begin{equation}\label{eq:MatrixVectorIneq}
\norm{v}_1 \le \norm{v}_0^{1/2}\norm{v}_2,\quad \norm{Av}_{\infty} \le \vertiii{A}_{\infty}\norm{v}_1,\quad\mbox{and}\quad |v^{\top}Av| \le \vertiii{A}_{\infty}\norm{v}_1^2.
\end{equation}
For any $1\le k\le p$, define the set of $k$-sparse submodels
\begin{equation}\label{eq:kSparse}
\mathcal{M}(k) := \{M:\, M\subseteq\{1,2,\ldots, p\},\,\, 1\le |M| \le k\},
\end{equation}
so that $\mathcal{M}(p)$ is the power set of $\{1,2,\ldots,p\}$ with the deletion of the empty set. Thus the set $\mathcal{M}(k)$ denotes the set of all non-empty submodels of size bounded by $k$. The most important aspect of our results is the ``uniform-in-submodel'' feature. These results are proved uniform over $M\in\mathcal{M}(k)$ for some $k$ that is allowed to diverge with $n$.

When fitting a linear regression, it is common to include an intercept term. To avoid extra notation, we assume that all covariates under consideration are included in the vectors $X_i$. So, take the first coordinate of all $X_i$'s to be 1, that is, $X_i(1) = 1$ for all $1\le i\le n,$ if an intercept is required. For any $M\subseteq\{1,2,\ldots,p\}$, define the ordinary least squares empirical risk (or objective) function as
\begin{equation}\label{eq:EmpObj}
\hat{R}_n(\theta; M) := \frac{1}{n}\sum_{i=1}^n \left\{Y_i - X_i^{\top}(M)\theta\right\}^2,\quad\mbox{for}\quad \theta\in\mathbb{R}^{|M|}.
\end{equation}
Expanding the square function it is clear that
\begin{equation}\label{eq:EmpObjExpanded}
\hat{R}_n(\theta; M) = \frac{1}{n}\sum_{i=1}^n Y_i^2 - \frac{2}{n}\sum_{i=1}^n Y_iX_i^{\top}(M)\theta + \theta^{\top}\left(\frac{1}{n}\sum_{i=1}^n X_i(M)X_i^{\top}(M)\right)\theta.
\end{equation}
Only the second and the third term depend on $\theta$. Because the quantities in these terms play a significant role in our analysis, define 
\begin{equation}\label{eq:MatrixVectorEstimators}
\begin{split}
\hat{\Sigma}_n := \frac{1}{n}\sum_{i=1}^n X_iX_i^{\top}\in\mathbb{R}^{p\times p},\quad&\mbox{and}\quad \hat{\Gamma}_n := \frac{1}{n}\sum_{i=1}^n X_iY_i\in\mathbb{R}^{p}.
\end{split}
\end{equation}
The least squares linear regression estimator $\hat{\beta}_{n,M}$ is defined as
\begin{equation}\label{eq:LeastSquaresEstimator}
\hat{\beta}_{n,M} ~:=~ \argmin_{\theta\in\mathbb{R}^{|M|}}\,\hat{R}_n(\theta; M) ~=~ \argmin_{\theta\in\mathbb{R}^{|M|}}\,\{\theta^{\top}\hat{\Sigma}_n(M)\theta - 2\theta^{\top}\hat{\Gamma}_n(M)\}.
\end{equation}
The notation $\argmin_{\theta}\,f(\theta)$ denotes the minimizer of $f(\theta)$. Based on the quadratic expansion \eqref{eq:EmpObjExpanded} of the empirical objective $\hat{R}_n(\theta; M)$, the estimator $\hat{\beta}_{n,M}$ is given by the closed form expression
\begin{equation}\label{eq:LeastSquaresEstimatorClosed}
\hat{\beta}_{n,M} = [\hat{\Sigma}_n(M)]^{-1}\hat{\Gamma}_n(M),
\end{equation}
assuming non-singularity of $\hat{\Sigma}_n(M)$. Note that $[\hat{\Sigma}_n(M)]^{-1}$ is \emph{not} equal to $\hat{\Sigma}_n^{-1}(M)$. The matrix $\hat{\Sigma}_n(M)$ being the average of $n$ rank one matrices in $\mathbb{R}^{|M|\times |M|}$, its rank is at most $\min\{|M|, n\}$. This implies that the least squares estimator $\hat{\beta}_{n,M}$ is not uniquely defined unless $|M| \le n$.

It is clear from \eqref{eq:LeastSquaresEstimatorClosed} that $\hat{\beta}_{n,M}$ is a smooth (non-linear) function of two averages $\hat{\Sigma}_n(M)$ and $\hat{\Gamma}_n(M)$. Assuming for a moment that the random vectors $(X_i, Y_i)$ are independent and identically distributed (iid) with finite fourth moments, it follows that $\hat{\Sigma}_n(M)$ and $\hat{\Gamma}_n(M)$ converge in probability to their expectations. The iid assumption here can be relaxed to weak dependence and non-identically distributed random vectors; see~\cite{White2001} for more details. 

Getting back to the general context, define the ``expected'' matrix and vector as
\begin{equation}\label{eq:MatrixVectorExpected}
\begin{split}
{\Sigma}_n := \frac{1}{n}\sum_{i=1}^n \mathbb{E}\left[X_iX_i^{\top}\right]\in\mathbb{R}^{p\times p},\quad&\mbox{and}\quad {\Gamma}_n := \frac{1}{n}\sum_{i=1}^n \mathbb{E}\left[X_iY_i\right]\in\mathbb{R}^{p}.
\end{split}
\end{equation}
Note that we write $\Sigma_n$ or $\Gamma_n$ (indexing by the sample size $n$) for two reasons. Firstly, we do not assume the random vectors are identically distributed, and hence the expected matrix changes with $n$ even if the dimension is fixed. Secondly, the dimension in our setting is allowed to change with $n$, and hence, even if the observations are identically distributed, the expectation matrix changes with the sample size.

To define a target vector that is being consistently estimated by $\hat{\beta}_{n,M}$, consider the following simple calculation in a simpler setting where $|M|$ does not change with $n$. As noted above $\hat{\beta}_{n,M} = [\hat{\Sigma}_n(M)]^{-1}\hat{\Gamma}_n(M)$, and if 
\begin{equation}\label{eq:ConsistencyAssumption}
(\hat{\Sigma}_n - \Sigma_n,\,\hat{\Gamma}_n - \Gamma_n)~\overset{P}{\to}~ 0\quad\mbox{as}\quad n\to\infty,
\end{equation}
then by a Slutsky type argument, it follows that
\begin{equation}\label{eq:DefiningTarget}
\hat{\beta}_{n,M} - \beta_{n,M} ~\overset{P}{\to}~ 0\quad\mbox{as}\quad n\to\infty,
\end{equation}
where
\begin{equation}\label{eq:LeastSquaresExpected}
\begin{split}
\beta_{n,M} ~&:=~ [\Sigma_n(M)]^{-1}\Gamma_n(M) ~=~ \argmin_{\theta\in\mathbb{R}^{|M|}}\, \{\theta^{\top}\Sigma_n(M)\theta - 2\theta^{\top}\Gamma_n(M)\}.
\end{split}
\end{equation} 
The convergence statement~\eqref{eq:DefiningTarget} only concerns a single submodel $M$ and is not uniform over $M$. By uniform-in-submodel $\norm{\cdot}_2$-norm consistency of $\hat{\beta}_{n,M}$ to $\beta_{n,M}$ for $M\in\mathcal{M}(k)$, we mean that
\[
\sup_{M\in\mathcal{M}(k)}\,\norm{\hat{\beta}_{n,M} - \beta_{n,M}}_2 = o_p(1)\quad\mbox{as}\quad n\to\infty.
\]

As shown above, convergence of $\hat{\beta}_{n,M}$ to $\beta_{n,M}$ only requires convergence of $\hat{\Sigma}_n(M)$ to $\Sigma_n(M)$ and $\hat{\Gamma}_n(M)$ to $\Gamma_n(M)$. It is not required that these matrices and vectors are averages of random matrices and random vectors.

In the following section, in proving deterministic inequalities, we generalize the linear regression estimator by the function $\beta_M:\mathbb{R}^{p\times p} \times \mathbb{R}^{p} \to \mathbb{R}^{|M|}$ as
\begin{equation}\label{eq:LinearRegressionMap}
\beta_M\left(\Sigma, \Gamma\right) = [\Sigma(M)]^{-1}\Gamma(M),
\end{equation}
assuming the existence of the inverse of $\Sigma(M)$. We call this $\beta_M(\cdot, \cdot)$ the \emph{linear regression map}. It is evident that
\[
\hat{\beta}_{n,M} ~=~ \beta_M(\hat{\Sigma}_n, \hat{\Gamma}_n)\quad\mbox{and}\quad \beta_{n,M} ~=~ \beta_M(\Sigma_n, \Gamma_n). 
\] 

There are many potential applications that require replacing the sample average matrices in the linear regression estimator by a suitable non-average version, e.g., shrinkage or robust estimators. Three of these applications are listed in Section~\ref{subsec:ApplicationsMap}. To distinguish the estimator $\hat{\beta}_{n,M}$ with sample averages from the {linear regression map}, we call $\hat{\beta}_{n,M}$ as the OLS estimator.

In the next section, we shall prove a bound of the type
\begin{equation}\label{eq:BoundSimulConsis}
\norm{\beta_M\left(\Sigma_1, \Gamma_1\right) - \beta_M\left(\Sigma_2, \Gamma_2\right)}_2 ~\le~ F_M\left(\Sigma_1 - \Sigma_2,\,\Gamma_1 - \Gamma_2\right)\quad{\mbox{for all}\quad M\in\mathcal{M}(k)}
\end{equation}
and for some function $F_M(\cdot, \cdot)$. Taking $(\Sigma_1, \Gamma_1) = (\hat{\Sigma}_n, \hat{\Gamma}_n)$ and $(\Sigma_2, \Gamma_2) = (\Sigma_n, \Gamma_n)$, inequality~\eqref{eq:BoundSimulConsis} is useful for the purpose of proving~\eqref{eq:SimulConsistency}. In regard to~\eqref{eq:BoundSimulConsis}, thinking of $\beta_M$ as a function of $(\Sigma, \Gamma)$, our results are essentially about studying Lipschitz continuity properties and understanding what kind of norms are best suited for this purpose. Using the smoothness of the linear regression map, we also obtain a bound on
\[
\sup_{M\in\mathcal{M}(k)}\,\norm{\beta_M(\Sigma_1, \Gamma_1) - \beta_M(\Sigma_2, \Gamma_2) - \nabla\beta_M(\Sigma_2, \Gamma_2)(\Sigma_1 - \Sigma_2, \Gamma_1 - \Gamma_2)}_2,
\]
where $\nabla\beta_M(\cdot, \cdot)$ represents the gradient of the linear regression map. The following error norms will be very useful for these results:
\begin{equation}\label{eq:ErrorNorms}
\begin{split}
\RIP(k, \Sigma_1 - \Sigma_2) &:= \sup_{M\in\mathcal{M}(k)}\,\norm{\Sigma_1(M) - \Sigma_2(M)}_{op},\\
\mathcal{D}\left(k, \Gamma_1 - \Gamma_2\right)  &= \sup_{M\in\mathcal{M}(k)}\,\norm{\Gamma_1(M) - \Gamma_2(M)}_2.
\end{split}
\end{equation}
The quantity $\RIP$ is a norm for any $k\ge 2$ and is not a norm for $k = 1$. This error norm is very closely related to the restricted isometry property used in the compressed sensing and high-dimensional linear regression literature where $\Sigma_2$ is the identity matrix. Also, define the $k$-sparse minimum singular value of a matrix $A\in\mathbb{R}^{p\times p}$ as
\begin{equation}\label{eq:MinimalkSparse}
\Lambda(k; A) = \inf_{\theta\in\mathbb{R}^p, \norm{\theta}_0 \le k}\,\frac{\norm{A\theta}_2}{\norm{\theta}_2}.
\end{equation}
Even though all the results in the next section are written in terms of the linear regression map \eqref{eq:LinearRegressionMap}, our main focus will still be the matrices and vectors defined in~\eqref{eq:MatrixVectorEstimators} and~\eqref{eq:MatrixVectorExpected}.
\section{Deterministic Results for Linear Regression}\label{sec:Deterministic}
\subsection{Can we expect deterministic inequalities?}\label{subsec:CanWeExpect}
Classical asymptotic theory for linear regression or for that matter any estimation problem usually starts with an assumption that the observations are independent or otherwise follow a specific stochastic dependence. What we are aiming for is a purely deterministic inequality that does not even assume randomness of the observations. 

To see whether we can at all expect a deterministic inequality, let us consider a simple example with only one submodel $M = \{1\}$, that is, a simple regression through the origin based on one regressor. For this case, let us write
\[
\hat{\sigma}_n^2 := \hat{\Sigma}_n(M),\quad\hat{\gamma}_n := \hat{\Gamma}_n(M),\quad\sigma_n^2 := \Sigma_n(M),\,\quad\mbox{and}\,\quad \gamma_n := \Gamma_n(M).
\]
Note that these are all scalar quantities. Now the regression estimator and targets become
\[
\hat{\beta}_{n,M} = \frac{\hat{\gamma}_n}{\hat{\sigma}_n^2}\quad\mbox{and}\quad \beta_{n,M} = \frac{\gamma_n}{\sigma_n^2}.
\]
Observe that
\begin{align*}
\left|\hat{\beta}_{n,M} - \beta_{n,M}\right| &= \left|\frac{\hat{\gamma}_n}{\hat{\sigma}_n^2} - \frac{\gamma_n}{\sigma_n^2}\right|\\ &\le \left|\frac{1}{\hat{\sigma}_n^2} - \frac{1}{\sigma_n^2}\right|\hat{\gamma}_n + \frac{1}{\sigma_n^2}\left|\hat{\gamma}_n - \gamma_n\right|\\
&\le \sigma_n^{-2}\left|\hat{\sigma}_n^2 - \sigma_n^2\right|\times|\hat{\beta}_{n,M}| + \sigma_n^{-2}\left|\hat{\gamma}_n - \gamma_n\right|\\
&\le \sigma_n^{-2}\left|\hat{\sigma}_n^2 - \sigma_n^2\right|\times|\hat{\beta}_{n,M} - \beta_{n,M}| + \sigma_n^{-2}\left|\hat{\sigma}_n^2 - \sigma_n^2\right|\times|\beta_{n,M}| + \sigma_n^{-2}\left|\hat{\gamma}_n - \gamma_n\right|.
\end{align*}
Solving this inequality for $|\hat{\beta}_{n,M} - \beta_{n,M}|$, we get
\[
\left|\hat{\beta}_{n,M} - \beta_{n,M}\right| \le \frac{\left|\hat{\sigma}_n^2 - \sigma_n^2\right|\times|\beta_{n,M}| + \left|\hat{\gamma}_n - \gamma_n\right|}{\sigma_n^2 - \left|\hat{\sigma}_n^2 - \sigma_n^2\right|}.
\]
This is a deterministic inequality that does not require any probabilistic structure on the data, and more importantly, the right hand side tends to zero if $\hat{\sigma}_n^2 - \sigma_n^2 = o(\sigma_n^2)$ and $\hat{\gamma}_n - \gamma_n {= o(\sigma_n^2)}$. Because this bound is a deterministic inequality, taking a supremum over a collection of submodels does not invalidate the inequality. This is \emph{not} the case if we only have an asymptotic result. All our deterministic inequalities to be stated/proved in the forthcoming sections are variations of the calculation above. One might suspect that the closed form expression of the linear regression map made a deterministic inequality possible, but as shown in~\cite{kuchibhotla2018Deter} most ``smooth'' $M$-estimators satisfy this type of result.
\subsection{Main Results}\label{subsec:MainResultsDeterministic}
All our results in this section depend on the error norms $\RIP(k, \Sigma_1 - \Sigma_2)$ and $\mathcal{D}(k, \Gamma_1 - \Gamma_2)$ in~\eqref{eq:ErrorNorms}. These are, respectively, the maximal $k$-sparse eigenvalue of $\Sigma_1 - \Sigma_2$ and the maximal $k$-sparse Euclidean norm of $\Gamma_1 - \Gamma_2$. At first glance, it may not be clear how these quantities behave. We first present a simple inequality for $\RIP$ and $\mathcal{D}$ in terms of $\vertiii{\cdot}_{\infty}$ and $\norm{\cdot}_{\infty}$.
\begin{prop}\label{prop:UniformNonSing}
For any $k\ge 1$,
\begin{align*}
\sup_{M\in\mathcal{M}(k)}\norm{\Sigma_1(M) - \Sigma_2(M)}_{op} &\le k\vertiii{\Sigma_1 - \Sigma_2}_{\infty},\\
\sup_{M\in\mathcal{M}(k)}\norm{\Gamma_1(M) - \Gamma_2(M)}_2 &\le k^{1/2}\norm{\Gamma_1 - \Gamma_2}_{\infty}.
\end{align*}
\end{prop}
\begin{proof}
See Appendix~\ref{appsec:propUniformNonSing} for a proof.
\end{proof}
In many cases, it is much easier to control the maximum elementwise norm rather than the $\RIP$ error norm. However, the factor $k$ on the right hand side often leads to sub-optimal dependence in the dimension. For the special cases of independent and dependent random vectors (to be discussed in Sections~\ref{sec:Independent} and~\ref{sec:Dependent}), we directly control $\RIP$ and $\mathcal{D}$.

The sequence of results to follow are related to uniform consistency in $\norm{\cdot}_2$- and $\norm{\cdot}_1$-norms. To state these results, we require the following quantities representing the strength of regression (or linear association). For $r, k \ge 1$
\begin{equation}\label{eq:StrengthS}
S_{r,k}(\Sigma, \Gamma) := \sup_{M\in\mathcal{M}(k)}\,\norm{\beta_M(\Sigma, \Gamma)}_r = \sup_{M\in\mathcal{M}(k)}\norm{[\Sigma(M)]^{-1}\Gamma(M)}_r.
\end{equation}
For the following theorem, recall the $k$-sparse minimum singular value $\Lambda(\cdot; \cdot)$ defined in~\eqref{eq:MinimalkSparse} and the error metrics defined in~\eqref{eq:ErrorNorms}.
\begin{thm}(Uniform $L_2$-consistency)\label{thm:L2UniformConsis}
Let $k \ge 1$ be any integer such that 
\begin{equation}\label{eq:RIPConverge}
\RIP(k, \Sigma_1 - \Sigma_2) \le \Lambda(k; \Sigma_2).
\end{equation} 
Then simultaneously for all $M\in\mathcal{M}(k)$,
\begin{equation}\label{eq:MarginalL2}
\begin{split}
\norm{\beta_M(\Sigma_1, \Gamma_1) \right.&-\left. \beta_{M}(\Sigma_2, \Gamma_2)}_2\le \frac{\mathcal{D}(k, \Gamma_1 - \Gamma_2) + \RIP(k, \Sigma_1 - \Sigma_2)\norm{\beta_{M}(\Sigma_2, \Gamma_2)}_2}{\Lambda(k; \Sigma_2) - \RIP(k, \Sigma_1 - \Sigma_2)}.
\end{split}
\end{equation}
\end{thm}
\begin{proof}
Recall from the linear regression map \eqref{eq:LinearRegressionMap} that
\[
{\beta}_M(\Sigma_1, \Gamma_1) = \left[\Sigma_1(M)\right]^{-1}\Gamma_1(M)\quad\mbox{and}\quad \beta_{M}(\Sigma_2, \Gamma_2) = \left[\Sigma_2(M)\right]^{-1}\Gamma_2(M).
\]
Fix $M\in\mathcal{M}(k)$. Then
\begin{align*}
\norm{\beta_M(\Sigma_1, \Gamma_1) - \beta_{M}(\Sigma_2, \Gamma_2)}_2 &= \norm{\left[\Sigma_1(M)\right]^{-1}\Gamma_1(M) - \left[\Sigma_2(M)\right]^{-1}\Gamma_2(M)}_2\\
&\le \norm{\left(\left[\Sigma_1(M)\right]^{-1} - \left[\Sigma_2(M)\right]^{-1}\right)\Gamma_1(M)}_2\\ &\qquad+ \norm{\left[\Sigma_2(M)\right]^{-1}\left(\Gamma_1(M) - \Gamma_2(M)\right)}_2\\
&=: \Delta_1 + \Delta_2.
\end{align*}
By definition of the operator norm, 
\begin{equation}\label{eq:BoundonDelta2}
\Delta_2 \le \left[\Lambda(k; \Sigma_2)\right]^{-1}\norm{\Gamma_1(M) - \Gamma_2(M)}_2 \le \left[\Lambda(k; \Sigma_2)\right]^{-1}\mathcal{D}\left(k, \Gamma_1 - \Gamma_2\right).
\end{equation}
To control $\Delta_1$, note that
\begin{align*}
\Delta_1 &\le \norm{\left(I_M - \left[\Sigma_2(M)\right]^{-1}\Sigma_1(M)\right)\left[\Sigma_1(M)\right]^{-1}\Gamma_1(M)}_2\\
&\le \norm{\left(I_M - \left[\Sigma_2(M)\right]^{-1}\Sigma_1(M)\right)}_{op}\norm{\beta_M(\Sigma_1, \Gamma_1)}_2\\
&\le \left[\Lambda(k; \Sigma_2)\right]^{-1}\norm{\Sigma_1(M) - \Sigma_2(M)}_{op}\norm{\beta_M(\Sigma_1, \Gamma_1)}_2\\
&\le \left[\Lambda(k; \Sigma_2)\right]^{-1}\RIP(k, \Sigma_1 - \Sigma_2)\norm{\beta_M(\Sigma_1, \Gamma_1)}_2,
\end{align*}
where $I_M$ represents the identity matrix of dimension $|M|\times|M|$. Now combining bounds on $\Delta_1, \Delta_2$, we get
\[
\norm{\beta_M(\Sigma_1, \Gamma_1) - \beta_M(\Sigma_2, \Gamma_2)}_2 \le \frac{\mathcal{D}(k, \Gamma_1 - \Gamma_2) + \RIP(k,\Sigma_1 - \Sigma_2)\norm{\beta_M(\Sigma_1, \Gamma_1)}_2}{\Lambda(k; \Sigma_2)}.
\]
{Subtracting and adding $\beta_{M}(\Sigma_2, \Gamma_2)$ from $\beta_M(\Sigma_1, \Gamma_1)$, we get
\begin{align*}
\norm{\beta_M(\Sigma_1, \Gamma_1) - \beta_M(\Sigma_2, \Gamma_2)}_2 &\le \frac{\mathcal{D}(k, \Gamma_1 - \Gamma_2) + \RIP(k,\Sigma_1 - \Sigma_2)\norm{\beta_M(\Sigma_2, \Gamma_2)}_2}{\Lambda(k; \Sigma_2)}\\ 
&\qquad+ \frac{\RIP(k, \Sigma_1 - \Sigma_2)}{\Lambda(k, \Sigma_2)}\|\beta_M(\Sigma_1, \Gamma_1) - \beta_M(\Sigma_2, \Gamma_2)\|_2.
\end{align*}
Solving this inequality under} assumption \eqref{eq:RIPConverge}, it follows for all $M\in\mathcal{M}(k)$ that
\[
\norm{\beta_M(\Sigma_1, \Gamma_1) - \beta_{M}(\Sigma_2, \Gamma_2)}_2 \le \frac{\mathcal{D}(k, \Gamma_1 - \Gamma_2) + \RIP(k, \Sigma_1 - \Sigma_2)\norm{\beta_{M}(\Sigma_2, \Gamma_2)}_2}{\Lambda(k; \Sigma_2) - \RIP(k; \Sigma_2)}.
\]
This proves the result.
\end{proof}
As will be seen in the application of Theorem \ref{thm:L2UniformConsis}, the complicated looking bound provided above gives the ``optimal'' bound. Combining Proposition \ref{prop:UniformNonSing} and Theorem \ref{thm:L2UniformConsis}, we get the following simple corollary that gives sub-optimal rates.
\begin{cor}\label{cor:L2UniformConsis}
Let $k \ge 1$ be any integer such that 
\[
k\vertiii{\Sigma_1 - \Sigma_2}_{\infty} \le \Lambda(k; \Sigma_2).
\]
Then
\begin{align*}
\sup_{M\in\mathcal{M}(k)}\,&\norm{\beta_M(\Sigma_1, \Gamma_1) - \beta_M(\Sigma_2, \Gamma_2)}_2\le \frac{k^{1/2}\norm{\Gamma_1 - \Gamma_2}_{\infty} + k\vertiii{\Sigma_1 - \Sigma_2}_{\infty}S_{2,k}(\Sigma_2, \Gamma_2)}{\Lambda(k; \Sigma_2) - k\vertiii{\Sigma_1 - \Sigma_2}_{\infty}}.
\end{align*}
\end{cor}
\begin{rem}\,(Bounding $S_{2,k}$ in~\eqref{eq:StrengthS})\label{rem:StrengthBounding}
The bound for uniform $L_2$-consistency requires a bound on~$\norm{\beta_M(\Sigma_2, \Gamma_2)}_2$ in  addition to bounds on the error norms related to $\Sigma$-matrices and $\Gamma$-vectors. It is a priori not clear how this quantity might vary as the dimension of the submodel $M$ changes. In the classical analysis of linear regression where a true linear model is assumed, the true parameter vector $\beta$ is seen as something chosen by nature and hence its norm is not under control of the statistician. {Hence}, in the classical analysis, a growth rate on $\norm{\beta}_2$ is imposed as an assumption. 

From the viewpoint taken in this paper, under misspecification  nature picks the whole distribution sequence of random vectors and hence the quantity $\beta_M(\cdot, \cdot)$ that came up in the analysis. In the full generality of linear regression maps considered here, we do not know of any techniques to bound the norm of this vector. It is, however, possible to bound it if $\beta_M(\cdot, \cdot)$ is defined by a least squares linear regression problem. Recall the definition of $\Sigma_n, \Gamma_n$ from~\eqref{eq:MatrixVectorExpected} and $\beta_{n,M}$ from~\eqref{eq:LeastSquaresExpected}. Observe that by definition of $\beta_{n,M}$,
\[
0 \le \frac{1}{n}\sum_{i=1}^n \mathbb{E}\left[\left\{Y_i - X_i^{\top}(M)\beta_{n,M}\right\}^2\right] \le \frac{1}{n}\sum_{i=1}^n \mathbb{E}\left[Y_i^2\right] - \beta_{n,M}^{\top}\Sigma_n(M)\beta_{n,M}.
\]
{This holds because $\beta_{n,M}$ satisfies $n^{-1}\sum_{i=1}^n \mathbb{E}[X_{i}(M)Y_i] = n^{-1}\sum_{i=1}^n \mathbb{E}[X_i(M)X_{i}^{\top}(M)\beta_{n,M}] = \Sigma_n(M)\beta_{n,M}$.}
Hence for every $M\in\mathcal{M}(p)$,
\[
\norm{\beta_{n,M}}_2^2\lambda_{\min}\left(\Sigma_n(M)\right) \le \beta_{n,M}\Sigma_n(M)\beta_{n,M} \le \frac{1}{n}\sum_{i=1}^n \mathbb{E}\left[Y_i^2\right].
\]
Therefore, using the definitions of $\Lambda(k; \Sigma_n)$ and $S_{r,k}$ in \eqref{eq:MinimalkSparse} and \eqref{eq:StrengthS},
\begin{equation}\label{eq:StrengthBounds}
\begin{split}
S_{2,k}(\Sigma_n, \Gamma_n) &\le\left(\frac{1}{n\Lambda(k; \Sigma_n)}\sum_{i=1}^n \mathbb{E}\left[Y_i^2\right]\right)^{1/2},\\
S_{1,k}(\Sigma_n, \Gamma_n) &\le \left(\frac{k}{n\Lambda(k; \Sigma_n)}\sum_{i=1}^n \mathbb{E}\left[Y_i^2\right]\right)^{1/2}.
\end{split}
\end{equation}
It is immediate from these results that if the second moment of the response is uniformly bounded, then $S_{2,k}$ behaves like a constant when $\Sigma_n$ is well-conditioned. See \cite{Barb11} for a similar calculation.
\end{rem}
Based on the uniform-in-submodel $\norm{\cdot}_2$-bound, the following result is trivially proved.
\begin{thm}(Uniform $L_1$-consistency)\label{thm:L1UniformConsis}
Let $k \ge 1$ be such that 
\begin{equation}\label{eq:RIPConvergeL1}
\RIP(k, \Sigma_1 - \Sigma_2) \le \Lambda(k; \Sigma_2).
\end{equation} 
Then simultaneously for all $M\in\mathcal{M}(k)$,
\begin{equation}\label{eq:MarginalL1}
\begin{split}
\norm{\beta_M(\Sigma_1, \Gamma_1) \right.&-\left. \beta_{M}(\Sigma_2, \Gamma_2)}_1\\ &\le |M|^{1/2}\frac{\mathcal{D}\left(k, \Gamma_1 - \Gamma_2\right) + \RIP(k, \Sigma_1 - \Sigma_2)\norm{\beta_{M}(\Sigma_2, \Gamma_2)}_2}{\Lambda(k; \Sigma_2) - \RIP(k, \Sigma_1 - \Sigma_2)}.
\end{split}
\end{equation}
\end{thm}
\begin{proof}
The proof follows by using the first inequality in \eqref{eq:MatrixVectorIneq}.
\end{proof}
The results above only prove a rate of convergence that gives uniform consistency. They are therefore not readily applicable for (asymptotic) inference. For inference about a parameter, an asymptotic distribution result is required, usually asymptotic normality, which is typically proved by way of an asymptotic linear representation. In what follows we derive a uniform-in-submodel linear representation for the linear regression map. The result in terms of the regression map itself is somewhat abstract, hence it might be helpful to revisit the usual estimators $\hat{\beta}_{n,M}$ and $\beta_{n,M}$ from~\eqref{eq:LeastSquaresEstimator} and~\eqref{eq:LeastSquaresExpected} to understand what kind of representation is possible. From the definition of $\hat{\beta}_{n,M}$, we have
\[
\hat{\Sigma}_n(M)\hat{\beta}_{n,M} = \hat{\Gamma}_{n}(M) \quad\Rightarrow\quad \hat{\Sigma}_n(M)\left(\hat{\beta}_{n,M} - \beta_{n,M}\right) = \hat{\Gamma}_{n}(M) - \hat{\Sigma}_n(M)\beta_{n,M}.
\]
Assuming $\hat{\Sigma}_n(M)$ and $\Sigma_n(M)$ are close, one would expect
\begin{equation}\label{eq:ApproxRepreError}
\norm{\hat{\beta}_{n,M} - \beta_{n,M} - \left[\Sigma_{n}(M)\right]^{-1}\left(\hat{\Gamma}_n(M) - \hat{\Sigma}_n(M)\beta_{n,M}\right)}_2 \approx 0.
\end{equation}
Note, by substituting all the definitions, that
\[
\left[\Sigma_{n}(M)\right]^{-1}\left(\hat{\Gamma}_n(M) - \hat{\Sigma}_n(M)\beta_{n,M}\right) = \frac{1}{n}\sum_{i=1}^n \left[\Sigma_n(M)\right]^{-1}X_i(M)(Y_i - X_i^{\top}(M)\beta_{n,M}).
\]
This being an average (a linear functional), the left hand side quantity in \eqref{eq:ApproxRepreError} is called the linear representation error. Now, using the same argument and substituting $\Sigma_1$ and $\Sigma_2$ for $\hat{\Sigma}_n$ and $\Sigma_n$, respectively we get the following result. Recall the notations $S_{2,k}(\cdot, \cdot)$ and $\Lambda(\cdot, \cdot)$ from Equations~\eqref{eq:StrengthS} and~\eqref{eq:MinimalkSparse}.
\begin{thm}(Uniform Linear Representation)\label{thm:UniformLinearRep}
Let $k \ge 1$ be any integer such that
\begin{equation}\label{eq:RIPSmall}
\RIP(k, \Sigma_1 - \Sigma_2) \le \Lambda(k; \Sigma_2).
\end{equation}
Then for all submodels $M\in\mathcal{M}(k)$,
\begin{equation}\label{eq:FirstPartLinearRep}
\begin{split}
\norm{\beta_M(\Sigma_1, \Gamma_1) \right.&-\left. \beta_M(\Sigma_2, \Gamma_2) - \left[\Sigma_2(M)\right]^{-1}\left(\Gamma_1(M) - \Sigma_1(M)\beta_M(\Sigma_2, \Gamma_2)\right)}_2\\
&\le \frac{\RIP(k, \Sigma_1 - \Sigma_2)}{\Lambda(k; \Sigma_2)}\norm{\beta_{M}(\Sigma_1, \Gamma_1) - \beta_M(\Sigma_2, \Gamma_2)}_2.
\end{split}
\end{equation}
Furthermore, using Theorem~\ref{thm:L2UniformConsis}, we get
\begin{equation}\label{eq:SecondPartLinearRep}
\begin{split}
\sup_{M\in\mathcal{M}(k)}&\norm{\beta_M(\Sigma_1, \Gamma_1) - \beta_M(\Sigma_2, \Gamma_2) - \left[\Sigma_2(M)\right]^{-1}\left(\Gamma_1(M) - \Sigma_1(M)\beta_M(\Sigma_2, \Gamma_2)\right)}_2\\
&\quad\le \frac{\RIP(k, \Sigma_1 - \Sigma_2)}{\Lambda(k; \Sigma_2)}\frac{\mathcal{D}\left(k, \Gamma_1 - \Gamma_2\right) + \RIP(k, \Sigma_1 - \Sigma_2)S_{2,k}(\Sigma_2, \Gamma_2)}{\Lambda(k; \Sigma_2) - \RIP(k, \Sigma_1 - \Sigma_2)}.
\end{split}
\end{equation}
\end{thm}
\begin{proof}
From the definition~\eqref{eq:LinearRegressionMap} of $\beta_M(\Sigma, \Gamma)$, we have
\begin{align}
\Sigma_1(M)\beta_M(\Sigma_1, \Gamma_1) - \Gamma_1(M) &= 0,\label{eq:FirstPair}\\
\Sigma_2(M)\beta_M(\Sigma_2, \Gamma_2) - \Gamma_2(M) &= 0.\label{eq:SecondPair}
\end{align}
Adding and subtracting $\beta_M(\Sigma_2, \Gamma_2)$ from $\beta_M(\Sigma_1, \Gamma_1)$ in~\eqref{eq:FirstPair}, it follows that
\[
\Sigma_1(M)\left(\beta_M(\Sigma_1, \Gamma_1) - \beta_M(\Sigma_2, \Gamma_2)\right) = \Gamma_1(M) - \Sigma_1(M)\beta_M(\Sigma_2, \Gamma_2).
\]
Now adding and subtracting $\Sigma_2(M)$ from $\Sigma_1(M)$ in this equation, we get
\begin{equation}\label{eq:MainEquation}
\begin{split}
&\left(\Sigma_2(M) - \Sigma_1(M)\right)\left(\beta_M(\Sigma_1, \Gamma_1) - \beta_M(\Sigma_2, \Gamma_2)\right)\\ 
&\qquad= \Sigma_2(M)\left(\beta_M(\Sigma_1, \Gamma_1) - \beta_M(\Sigma_2, \Gamma_2)\right) - \left[\Gamma_1(M) - \Sigma_1(M)\beta_M(\Sigma_2, \Gamma_2)\right].
\end{split}
\end{equation}
The right hand side is almost the quantity we need to bound to establish the result. Multiplying both sides of the equation by $[\Sigma_2(M)]^{-1}$ and then applying the Euclidean norm implies that for $M\in\mathcal{M}(k)$,
\begin{align*}
\norm{\beta_M(\Sigma_1, \Gamma_1) \right.&-\left. \beta_M(\Sigma_2, \Gamma_2) - \left[\Sigma_2(M)\right]^{-1}\left\{\Gamma_1(M) - \Sigma_1(M)\beta_M(\Sigma_2, \Gamma_2)\right\}}_2\\
&\le \frac{\norm{\Sigma_1(M) - \Sigma_2(M)}_{op}}{\Lambda(k; \Sigma_2)}\norm{\beta_M(\Sigma_1, \Gamma_1) - \beta_M(\Sigma_2, \Gamma_2)}_2.
\end{align*}
This proves the first part of the result. The second part of the result follows by the application of Theorem~\ref{thm:L2UniformConsis}.
\end{proof}
\begin{rem}\,(Matching Lower Bounds)
The bound~\eqref{eq:FirstPartLinearRep} only proves an upper bound. It can, however, be seen from Equation~\eqref{eq:MainEquation} that for any $M\in\mathcal{M}(k)$,
\begin{equation*}
\begin{split}
\norm{\beta_M(\Sigma_1, \Gamma_1) \right.&-\left. \beta_M(\Sigma_2, \Gamma_2) - \left[\Sigma_2(M)\right]^{-1}\left(\Gamma_1(M) - \Sigma_1(M)\beta_M(\Sigma_2, \Gamma_2)\right)}_2\\ &\quad= \norm{\left[\Sigma_2(M)\right]^{-1}\left(\Sigma_1(M) - \Sigma_2(M)\right)\left(\beta_M(\Sigma_1, \Gamma_1) - \beta_M(\Sigma_2, \Gamma_2)\right)}_2\\&\quad\ge {C_*(k, \Sigma_2)}{\Lambda(k, \Sigma_1 - \Sigma_2)}\norm{\beta_M(\Sigma_1, \Gamma_1) - \beta_M(\Sigma_2, \Gamma_2)}_2,
\end{split}
\end{equation*}
where
\[
C_*(k, \Sigma_2) := \min_{M\in\mathcal{M}(k)}\lambda_{\min}\left(\left[\Sigma_2(M)\right]^{-1}\right) = \left[\RIP(k, \Sigma_2)\right]^{-1}.
\]
Recall from Equations~\eqref{eq:ErrorNorms} and~\eqref{eq:MinimalkSparse}, that
\[
\RIP(k, \Sigma_2) = \sup_{M\in\mathcal{M}(k)}\norm{\Sigma_2(M)}_{op}\quad\mbox{and}\quad \Lambda(k, \Sigma_1 - \Sigma_2) = \inf_{\theta\in\mathbb{R}^p, \norm{\theta}_0 \le k}\,{\frac{\|(\Sigma_1 - \Sigma_2)\theta\|_2}{\norm{\theta}_2}}.
\]
If the minimal and maximal $k$-sparse singular values of $\Sigma_1 - \Sigma_2$ are of the same order, then the upper and lower bounds for the linear representation error match up to the order under the additional assumption that the minimal and maximal sparse eigenvalues of $\Sigma_2$ are of the same order.
\end{rem}
\begin{rem}\,(Improved $\norm{\cdot}_2$-Error Bounds)
Uniform linear representation error bounds~\eqref{eq:FirstPartLinearRep} and~\eqref{eq:SecondPartLinearRep} prove more than just a linear representation. These bounds allow us to improve the bounds provided for uniform $L_2$-consistency. Bound~\eqref{eq:FirstPartLinearRep} is of the form
\[
\norm{u - v}_2 \le \delta\norm{u}_2\quad\Rightarrow\quad (1 - \delta)\norm{u}_2 \le \norm{v}_2 \le (1 + \delta)\norm{u}_2. 
\]
Therefore, assuming $\RIP(k, \Sigma_1 - \Sigma_2) \le \Lambda(k; \Sigma_2)/2$, it follows that for all $M\in\mathcal{M}(k)$,
\begin{equation}\label{eq:ImprovedL2}
\begin{split}
\frac{1}{2}&\norm{\left[\Sigma_2(M)\right]^{-1}\left(\Gamma_1(M) - \Sigma_1(M)\beta_M(\Sigma_2, \Gamma_2)\right)}_2\\ &\qquad\le \norm{\beta_M(\Sigma_1, \Gamma_1) - \beta_M(\Sigma_2, \Gamma_2)}_2\\ &\qquad\qquad\le 2\norm{\left[\Sigma_2(M)\right]^{-1}\left(\Gamma_1(M) - \Sigma_1(M)\beta_M(\Sigma_2, \Gamma_2)\right)}_2.
\end{split}
\end{equation}
This is a more precise result than informed by Theorem~\ref{thm:L2UniformConsis} because here we characterize the estimation error exactly up to a factor of 2. Also, note that in case of the least squares estimator and target, $\hat{\beta}_{n,M}$ and $\beta_{n,M}$, the upper and lower bounds here are Euclidean norms of averages of random vectors. Dealing with linear functionals like averages is much simpler than dealing with non-linear functionals such as $\hat{\beta}_{n,M}$.

If $\RIP(k, \Sigma_1 - \Sigma_2)$ converges to zero, then the right hand side of bound~\eqref{eq:FirstPartLinearRep} is of smaller order than both the terms appearing on the left hand side (which are the same as those appearing in~\eqref{eq:ImprovedL2}). This means that the linear representation error is of strictly smaller order than the estimator error simultaneously over all $M\in \mathcal{M}(k)$.
\end{rem}
\begin{rem}\,(Alternative to $\RIP$)
A careful inspection of the proof of Theorem~\ref{thm:UniformLinearRep} and Theorem~\ref{thm:L2UniformConsis} reveals that the bounds can be written in terms of 
\[
\sup_{M\in\mathcal{M}(k)}\,\norm{\left[\Sigma_2(M)\right]^{-1/2}{\Sigma}_1(M)\left[\Sigma_2(M)\right]^{-1/2} - I_{|M|}}_{op},
\]
instead of $\RIP(k, \Sigma_1 - \Sigma_2)$. Here $I_{|M|}$ is the identity matrix in $\mathbb{R}^{|M|\times|M|}$. Bounding this quantity might not require a bounded condition number of $\Sigma_2$; however, we will only deal with $\RIP(k, \Sigma_1 - \Sigma_2)$ in the following sections for convenience.
\end{rem}
Summarizing all the results in this section it is sufficient to control
\[
\RIP(k, \Sigma_1 - \Sigma_2)\quad\mbox{and}\quad \mathcal{D}\left(k, \Gamma_1 - \Gamma_2\right)
\]
to derive uniform-in-submodel results in any linear regression type problem. In this respect, these are the norms in which one should measure the accuracy of the Gram matrix and the inner product of covariates and response. Hence if one wishes to use shrinkage estimators, for example, because $\Sigma$ and $\Gamma$ are high-dimensional ``objects'', then the estimation accuracy should be measured with respect to $\RIP$ and $\mathcal{D}$ for uniform-in-submodel type results.

\subsection{Applications of the Linear Regression Map}\label{subsec:ApplicationsMap}
Before proceeding to the rates of convergence of these error norms for independent and dependent data, we describe the importance of defining the linear regression map with general matrices instead of just Gram matrices. The generality achieved so far would be worthless if no interesting applications existed. The goal now is to provide a few such interesting examples. 
\begin{enumerate}
\item \emph{Heavy-Tailed Observations:} The $\RIP(\cdot, \cdot)$-norm is a supremum over all submodels of size $k$ or less, hence the supremum is over
\[
\sum_{s = 1}^k \binom{p}{s} \le \sum_{s = 1}^k \frac{p^s}{s!} = \sum_{s = 1}^k \frac{k^s}{s!}\left(\frac{p}{k}\right)^s \le \left(\frac{ep}{k}\right)^k
\]
number of submodels. This bound is polynomial in the total number of covariates but is exponential in the size of the largest submodel under consideration. Therefore, if the total number of covariates $p$ is allowed to diverge, then the question we are interested in is inherently high-dimensional. If the usual Gram matrices are used then
\[
\RIP(k, \hat{\Sigma}_n - \Sigma_n) = \sup_{|M| \le k}\,\norm{\hat{\Sigma}_n(M) - \Sigma_n(M)}_{op},
\]
hence, $\RIP$ in this case is the supremum in the order of $(ep/k)^k$ many averages. As is well-understood from the literature on concentration of measure or even the union bound, one would require exponential tails on the initial random vectors to allow a good control on $\RIP(\cdot, \cdot)$ if the usual Gram matrix is used. Does this mean that the situation is hopeless if the initial random vectors do not have exponential tails? The short answer is ``not necessarily.'' Viewing the matrix $\Sigma_n$ (the ``population'' Gram matrix) as a target, there have been many variations of sample mean Gram matrix estimators that are shown to provide exponential tails even though the initial observations are heavy tailed. See, for example, \cite{Catoni12}, \cite{Wei17} and \cite{Catoni17}, along with the references therein, for more details on a specific estimator and its properties. It should be noted that these authors do not study the estimator accuracy with respect to the RIP-norm. 
\item \emph{Outlier Contamination:} Real data, more often than not, is contaminated with outliers, and it is a difficult problem to remove or downweight observations when contamination is present. Robust statistics provide estimators that can ignore or downweight the observations suspected to be outliers and yet perform comparably when there is no contamination present in the data. Some simple examples include entry-wise medians or trimmed means. See \cite{Minsker15} and references therein for some more examples. Almost none of these estimators are simple averages but behave regularly in the sense that they can be expressed as averages up to a negligible asymptotic remainder term. \cite{Chen13} provide a simple estimator of the Gram matrix under adversarial corruption and case-wise contamination.
\item \emph{Indirect Observations:} This example is taken from \cite{Loh12}. The setting is as follows. Instead of observing the real random vectors $(X_1, Y_1)$, $\ldots$, $(X_n, Y_n)$, we observe a sequence $(Z_1, Y_1), \ldots, (Z_n, Y_n)$ with $Z_i$ linked with $X_i$ via some conditional distribution that is for $1\le i\le n$,
\[
Z_i\sim Q(\cdot|X_i).
\]
As discussed on page 4 of \cite{Loh12}, this setting includes some interesting cases such as missing data and noisy covariates. A brief hint of the settings is given below:
\begin{itemize}
\item[--] If $Z_i = X_i + W_i$ where $W_i$ is independent of $X_i$ and has mean zero with a known covariance matrix.
\item[--] For some fraction $\rho\in[0, 1)$, we observe a random vector $Z_i\in\mathbb{R}^p$ such that for each component $j$, we independently observe $Z_{i}(j) = X_{i}(j)$ with probability $1 - \rho$ and $Z_i(j) = *$ with probability $\rho$. (Here $*$ means a missing value.)
\item[--] If $Z_i = X_i\odot u_i,$ where $u_i\in\mathbb{R}^p$ is again a random vector independent of $X_i$ and $\odot$ is the Hadamard (coordinate-wise) product. The problem of missing data is a special case.
\end{itemize}
On page 6, \cite{Loh12} provide various estimators in place of $\hat{\Sigma}_n$ in~\eqref{eq:MatrixVectorEstimators}. The assumption in Lemma 12 of \cite{Loh12} is essentially a bound on the $\RIP$-norm in our notation, and they verify this assumption in all the examples above. Hence all our results in this section apply to these settings. 
\end{enumerate}
\subsection{Application of Deterministic Inequalities to OLS}
In the following two sections, we prove finite sample non-asymptotic bounds for $\RIP(k, \Sigma_1 - \Sigma_2)$ and $\mathcal{D}(k, {\Gamma_1 - \Gamma_2})$ when 
$$\Sigma_1 = \hat{\Sigma}_n,\quad \Sigma_2 = \Sigma_n\quad\mbox{and}\quad \Gamma_1 = \hat{\Gamma}_n,\quad \Gamma_2 = \Gamma_n.$$ 
See Equations \eqref{eq:MatrixVectorEstimators} and \eqref{eq:MatrixVectorExpected}. For convenience, we rewrite Theorem \ref{thm:UniformLinearRep} for this setting. Also, for notational simplicity, let
\begin{equation}\label{eq:SimplifiedNotation}
\Lambda_n(k) := \Lambda(k, \Sigma_n),\,\, \RIP_n(k) := \RIP(k, \hat{\Sigma}_n - \Sigma_n)\quad\mbox{and}\quad\mathcal{D}_n(k) := \mathcal{D}(k, \hat{\Gamma}_n - \Gamma_n).
\end{equation}
Recall the definition of $\hat{\beta}_{n,M}$, $\beta_{n,M}$ and $S_{2,k}$ from~\eqref{eq:LeastSquaresEstimatorClosed}, \eqref{eq:LeastSquaresExpected} and~\eqref{eq:StrengthS}.
\begin{thm}\label{thm:SimplifiedLinearRep}
Let $k\ge 1$ be any integer such that $\RIP_n(k) \le \Lambda_n(k)$. Then for all submodels $M\in\mathcal{M}(k)$,
\begin{align*}
\sup_{M\in\mathcal{M}(k)}&\norm{\hat{\beta}_{n,M} - \beta_{n,M} - \frac{1}{n}\sum_{i=1}^n \left[\Sigma_n(M)\right]^{-1}X_i(M)\left(Y_i - X_i^{\top}(M)\beta_{n,M}\right)}_2\\
&\le \frac{\RIP_n(k)}{\Lambda_n(k)}\left(\frac{\mathcal{D}_n(k) + \RIP_n(k)S_{2,k}(\Sigma_n, \Gamma_n)}{\Lambda_n(k) - \RIP_n(k)}\right).
\end{align*}
\end{thm}
Recall here that $\Gamma_n$ and $\Sigma_n$ are \textbf{non-random} vectors/matrices given in~\eqref{eq:MatrixVectorExpected}. So Theorem~\ref{thm:SimplifiedLinearRep} (which is still a deterministic inequality) can be used to prove an asymptotic uniform linear representation. 
\begin{rem}\,(Non-uniform Bounds)
The bound above applies for any $k$ satisfying the assumption $\RIP_n(k) \le \Lambda_n(k)$. Noting that for $M\in\mathcal{M}(k)$, $\RIP_n(|M|) \le \RIP_n(k)$ as well as $\Lambda_n(|M|) \ge \Lambda_n(k)$, Theorem~\ref{thm:SimplifiedLinearRep} implies that
\begin{align*}
&\norm{\hat{\beta}_{n,M} - \beta_{n,M} - \frac{1}{n}\sum_{i=1}^n \left[\Sigma_n(M)\right]^{-1}X_i(M)\left(Y_i - X_i^{\top}(M)\beta_{n,M}\right)}_2\\ 
&\qquad\le \frac{\RIP_n(|M|)}{\Lambda_n(|M|)}\left(\frac{\mathcal{D}_n(|M|) + \RIP_n(|M|)S_{2,|M|}(\Sigma_n, \Gamma_n)}{\Lambda_n(|M|) - \RIP_n(|M|)}\right).
\end{align*}
The point made here is that even though the bound in Theorem~\ref{thm:SimplifiedLinearRep} only uses the maximal submodel size, it can recover submodel size dependent bounds because the result is proved for every $k$.
\end{rem}
\begin{rem}\,(Post-selection Consistency)
One of the main aspects of our results is in proving consistency of the least squares linear regression estimator after data exploration. Suppose a random submodel $\hat{M}$ chosen based on data satisfies $|\hat{M}| \le k$ with probability converging to one, that is, $\mathbb{P}(\hat{M}\in\mathcal{M}(k)) \to 1$. Then, with probability converging to one,
\[
\norm{\hat{\beta}_{n,\hat{M}} - \beta_{n,\hat{M}}}_2 \le \sup_{M\in\mathcal{M}(k)}\norm{\hat{\beta}_{n,M} - \beta_{n,M}}_2.
\]
A similar bound also holds for the linear representation error. Therefore, the uniform-in-submodel results above allow us to prove consistency and asymptotic normality of the least squares linear regression estimator after data exploration. See \cite{Belloni13} for related applications and methods of choosing the random submodel $\hat{M}$.
\end{rem}
\begin{rem}\,(Bounding $S_{2,k}$)
As shown in Remark~\ref{rem:StrengthBounding}, for the setting of averages
\begin{equation}\label{eq:S2kBounded}
S_{2,k}(\Sigma_n, \Gamma_n) \le \left(\frac{1}{n\Lambda_n(k)}\sum_{i=1}^n \mathbb{E}\left[Y_i^2\right]\right)^{1/2}.
\end{equation}
The quantity on the right hand side of~\eqref{eq:S2kBounded} is of the order $\Lambda_n^{-1/2}(k)$ under the assumption of bounded second moments of the $Y_i$'s. Therefore, we will not further write $S_{2,k}$ explicitly and just use $\Lambda_n^{-1/2}(k)$ instead.
\end{rem}
\section{Rates for Independent Observations}\label{sec:Independent}
In this section, we derive bounds for $\RIP_n(k)$ and $\mathcal{D}_n(k)$ defined in~\eqref{eq:SimplifiedNotation} under the assumption of independence and weak exponential tails. The setting is as follows. Suppose $(X_1, Y_1), \ldots, (X_n, Y_n)$ are a sequence of independent random vectors in $\mathbb{R}^{p}\times\mathbb{R}$. Consider the following assumptions:
\begin{description}
\item[\namedlabel{eq:Marginal}{(MExp)}] Assume that there exist positive numbers $\alpha > 0$, and $K_{n,p} > 0$ such that
\[
\max_{1\le j\le p}\max\left\{\norm{X_i(j)}_{\psi_{\alpha}}, \norm{Y_i}_{\psi_{\alpha}}\right\} \le K_{n,p}\quad\mbox{for all}\quad 1\le i\le n.
\]
\item[\namedlabel{eq:Joint}{(JExp)}] Assume that there exist positive numbers $\alpha > 0$, and $K_{n,p} > 0$ such that
\[
\max\left\{\norm{X_i^{\top}\theta}_{\psi_{\alpha}}, \norm{Y_i}_{\psi_{\alpha}}\right\} \le K_{n,p}\quad\mbox{for all}\quad \theta\in\mathbb{R}^p,\, \norm{\theta}_2 \le 1,\, 1\le i\le n.
\]
\end{description}
Recall that $X_i(j)$ means the $j$-th coordinate of $X_i$. The notation $\norm{\cdot}_{\psi_{\alpha}}$ refers to a quasi-norm defined by
\[
\norm{W}_{\psi_{\alpha}} := \inf\left\{C > 0:\, \mathbb{E}\left[\exp\left(\frac{|W|^{\alpha}}{C^{\alpha}}\right)\right] \le 2\right\},
\]
for any random variable $W$. Random variables $W$ satisfying $\norm{W}_{\psi_{\alpha}} < \infty$ are referred to as sub-Weibull of order $\alpha$, because $\norm{W}_{\psi_{\alpha}} < \infty$ implies that for all $t \ge 0,$
\[
\mathbb{P}\left(|W| \ge t\right) \le 2\exp\left(-\frac{t^{\alpha}}{\norm{W}_{\psi_{\alpha}}^{\alpha}}\right),
\]
where the right hand side resembles the survival function of a Weibull random variable of order $\alpha > 0$ (see~\cite{KuchAbhi17} for more details). The special cases $\alpha = 1, 2$ are very much used in the  high-dimensional literature as assumed tail behaviors. A random variable $W$ satisfying $\norm{W}_{\psi_{\alpha}} < \infty$ with $\alpha = 2$ is called sub-Gaussian, and with $\alpha = 1$ it is called sub-exponential (see \cite{VdvW96} for more details).

It is easy to see that Assumption~\ref{eq:Joint} implies Assumption~\ref{eq:Marginal}. We refer to Assumption~\ref{eq:Marginal} as a marginal assumption and Assumption~\ref{eq:Joint} as a joint assumption. It should be noted that Assumption~\ref{eq:Joint} is much stronger than~\ref{eq:Marginal} because~\ref{eq:Joint} implies that the coordinates of $X_i$ should be ``almost'' independent (see Chapter 3 of \cite{Vershynin18} and~\cite{KuchAbhi17} for further discussion).

The following results bound $\mathcal{D}_n(k)$ and $\RIP_n(k)$ based on Theorem~\ref{thm:IndependenceTailBound} in Appendix~\ref{AppSec:Independence}. Because $\RIP_n(k)$ involves operator norms over $k$-sparse unit balls, we will bound it using $\varepsilon$-nets for the union of these unit balls. This will also be useful for bounding $\mathcal{D}_n(k)$. Before stating the results, we need the following preliminary calculations and notations. For any set $K$ with metric $d(\cdot, \cdot)$, a set $\mathcal{N}$ is called a $\gamma$-net of $K$ with respect to $d$ if $\mathcal{N}\subset K$ and for any $z\in K$ there exists an $x\in\mathcal{N}$ such that $d(x,z) \le \gamma$. Let $\norm{\cdot}_2$ denote the Euclidean norm and define the $d$-dimensional unit ball by
\[
\mathcal{B}_{2,d} := \left\{x\in\mathbb{R}^d:\,\norm{x}_2 \le 1\right\}.
\]
Let $\mathcal{N}_d(\varepsilon)$ represent an $\varepsilon$-net of $\mathcal{B}_{2,d}$ with respect to the Euclidean norm. Define the $k$-sparse subset of the unit ball in $\mathbb{R}^p$ as
\begin{equation}\label{eq:kSparseSubsetDef}
\Theta_k := \left\{\theta\in\mathbb{R}^p:\,\norm{\theta}_0 \le k,\, \norm{\theta}_2 \le 1\right\}.
\end{equation}
With some abuse of notation, a disjoint decomposition of $\Theta_k$ can be written as
\[
\Theta_k = \bigcup_{s = 1}^k \bigcup_{|M| = s}\mathcal{B}_{2,s}.
\]
The last union includes repetition of $\mathcal{B}_{2,s}$ as subsets of $\mathbb{R}^p$ with unequal supports. Using this decomposition, it follows that a $1/4$-net $\mathcal{N}(\varepsilon, \Theta_k)$ of $\Theta_k$ with respect to the Euclidean norm on $\mathbb{R}^p$ can be chosen to satisfy
\[
\mathcal{N}\left(\varepsilon, \Theta_k\right) \subseteq \bigcup_{s = 1}^k \bigcup_{|M| = s} \mathcal{N}_s(\varepsilon),
\]
and hence, can be bounded in cardinality by
\[
\left|\mathcal{N}(\varepsilon, \Theta_k)\right| \le \sum_{s = 1}^k \binom{p}{s} \left|\mathcal{N}_{s}(\varepsilon)\right|.
\]  
Applying Lemma 4.1 of \cite{Pollard90} it follows that
\[
|\mathcal{N}_s(\varepsilon)| \le (1 + \varepsilon^{-1})^{s}\quad\Rightarrow\quad \left|\mathcal{N}(\varepsilon, \Theta_k)\right| \le \sum_{s = 1}^k \binom{p}{s}(1 + \varepsilon^{-1})^s \le \left(\frac{(1 + \varepsilon^{-1})ep}{k}\right)^k.
\]
(Lemma 4.1 of \cite{Pollard90} provides the bound on the covering number to be $(3/\varepsilon)^d$ but it can be improved from the proof to $(1 + 1/\varepsilon)^{d}$.) Here one can choose the elements of the covering set $\mathcal{N}_s(\varepsilon)$ to be $s$-sparse in $\mathbb{R}^p.$ See Lemma 3.3 of \cite{Plan13} for a similar result. Based on these calculations and the covering set $\mathcal{N}(\varepsilon, \Theta_k)$, we bound $\mathcal{D}_n(k)$ and $\RIP_n(k)$ by a finite maximum of mean zero averages.

Observe that
\begin{align*}
\mathcal{D}_n(k) &= \sup_{\theta\in\Theta_k}\,\theta^{\top}\left(\hat{\Gamma}_n - \Gamma_n\right)\\
&\le \sup_{\alpha\in\mathcal{N}(1/2, \Theta_k)}\, \alpha^{\top}\left(\hat{\Gamma}_n - \Gamma_n\right) + \sup_{\beta\in\Theta_k/2}\,\beta^{\top}\left(\hat{\Gamma}_n - \Gamma_n\right)\\
&= \sup_{\alpha\in\mathcal{N}(1/2, \Theta_k)}\, \alpha^{\top}\left(\hat{\Gamma}_n - \Gamma_n\right) + \frac{1}{2}\sup_{\beta\in\Theta_k}\,\beta^{\top}\left(\hat{\Gamma}_n - \Gamma_n\right).
\end{align*}
Therefore,
\begin{equation}\label{eq:FiniteDnk}
\mathcal{D}_n(k) \le 2\sup_{\theta\in\mathcal{N}(1/2,\Theta_k)}\,\left|\frac{1}{n}\sum_{i=1}^n \left\{\theta^{\top}X_iY_i - \mathbb{E}\left[\theta^{\top}X_iY_i\right]\right\}\right|.
\end{equation}
It is clear that the bound is sharp up to a constant factor. By a similar calculation, it can be shown that
\begin{equation}\label{eq:FiniteRIPnk}
\RIP_n(k) \le 2\sup_{\theta\in\mathcal{N}(1/4,\Theta_k)}\,\left|\frac{1}{n}\sum_{i=1}^n \left\{\left(X_i^{\top}\theta\right)^2 - \mathbb{E}\left[\left(X_i^{\top}\theta\right)^2\right]\right\}\right|.
\end{equation}
See Lemma 2.2 of \cite{Ver12} for a derivation. Importantly independence of the random vectors is not used in any of these calculations. Replacing the continuous supremum by a finite maximum works irrespective of how the random vectors are distributed. 

As an immediate corollary, we get the following rate of convergence results.
\begin{thm}\label{cor:RatesofConvergence}
Define for $k\ge 1$
\[
\Upsilon_{n,k}^{\Gamma} := \sup_{\theta\in\Theta_k}\,\frac{1}{n}\sum_{i=1}^n \mbox{Var}\left(\theta^{\top}X_iY_i\right),\quad\mbox{and}\quad \Upsilon_{n,k}^{\Sigma} := \sup_{\theta\in\Theta_k}\,\frac{1}{n}\sum_{i=1}^n \mbox{Var}\left(\left(\theta^{\top}X_i\right)^2\right).
\]
Then the following rates of convergence hold if $K_{n,p} = O(1)$.
\begin{enumerate}
\item[(a)] Under Assumption~\ref{eq:Marginal},
\begin{align*}
\mathcal{D}_n(k) &= O_p\left(\sqrt{\frac{\Upsilon_{n,k}^{\Gamma}k\log(ep/k)}{n}} + \frac{k^{1/2}(\log n)^{2/\alpha}(k\log(ep/k))^{1/T_1(\alpha/2)}}{n}\right),\\
\mathcal{\RIP}_n(k) &= O_p\left(\sqrt{\frac{\Upsilon_{n,k}^{\Sigma}k\log(ep/k)}{n}} + \frac{k(\log n)^{2/\alpha}(k\log(ep/k))^{1/T_1(\alpha/2)}}{n}\right).
\end{align*}
Here $T_1(\alpha) = \min\{\alpha, 1\}$.
\item[(b)] Under Assumption~\ref{eq:Joint},
\begin{align*}
\mathcal{D}_n(k) &= O_p\left(\sqrt{\frac{\Upsilon_{n,k}^{\Gamma}k\log(ep/k)}{n}} + \frac{(\log n)^{2/\alpha}(k\log(ep/k))^{1/T_1(\alpha/2)}}{n}\right),\\
\mathcal{\RIP}_n(k) &= O_p\left(\sqrt{\frac{\Upsilon_{n,k}^{\Sigma}k\log(ep/k)}{n}} + \frac{(\log n)^{2/\alpha}(k\log(ep/k))^{1/T_1(\alpha/2)}}{n}\right).
\end{align*}
\end{enumerate}
\end{thm}
For simplicity, we provide here only rates of convergence. A more precise tail bound is given in Theorem~\ref{thm:ExactTailIndependence} of Appendix~\ref{AppSec:Independence}.
\begin{rem}\,(Simplified Rates of Convergence)
In most cases the second term in the rate of convergence is of lower order than the first term. Hence, under both the assumptions~\ref{eq:Marginal} and~\ref{eq:Joint}, we get
\[
\mathcal{D}_n(k) = O_p\left(\sqrt{\frac{\Upsilon_{n,k}^{\Gamma}k\log(ep/k)}{n}}\right)\quad\mbox{and}\quad \RIP_n(k) = O_p\left(\sqrt{\frac{\Upsilon_{n,k}^{\Sigma}k\log(ep/k)}{n}}\right).
\]
We believe these to be optimal because if $X$ and $Y$ are independent and jointly Gaussian, then the rates would be $\sqrt{k\log(ep/k)/n}$; see Theorem 3.3 of \cite{Cai12} and Lemma 15 of \cite{Loh12} for related results.
\end{rem}
A direct application of Theorem~\ref{cor:RatesofConvergence} to Theorem~\ref{thm:SimplifiedLinearRep} implies the following uniform linear representation result for linear regression under independence. Recall the notation $\Lambda_n(k)$ from~\eqref{eq:SimplifiedNotation} and also $\hat{\beta}_{n,M}, \beta_{n,M}$ from~\eqref{eq:LeastSquaresEstimatorClosed}, \eqref{eq:LeastSquaresExpected}.
\begin{thm}\label{thm:LinearRepIndependence}
If $(\Lambda_n(k))^{-1} = O(1)$ as $n, p\to\infty$, then the following rates of convergence hold as $n\to\infty.$
\begin{enumerate}
\item[(a)] under Assumption~\ref{eq:Marginal}, 
\begin{align*}
\sup_{M\in\mathcal{M}(k)}&\norm{\hat{\beta}_{n,M} - \beta_{n,M}}_2\\
&= O_p\left(\sqrt{\frac{\Upsilon_{n,k}^{\Gamma}k\log(ep/k)}{n}} + K_{n,p}^2\frac{k(\log n)^{2/\alpha}(k\log(ep/k))^{1/T_1(\alpha/2)}}{n}\right),
\end{align*}
and
\begin{align*}
&\sup_{M\in\mathcal{M}(k)}\norm{\hat{\beta}_{n,M} - \beta_{n,M} - \frac{1}{n}\sum_{i=1}^n \left[\Sigma_n(M)\right]^{-1}X_i(M)\left(Y_i - X_i^{\top}(M)\beta_{n,M}\right)}_2\\
&\quad= O_p\left(\frac{\max\{\Upsilon_{n,k}^{\Gamma}, \Upsilon_{n,k}^{\Sigma}\}k\log(ep/k)}{n} + K_{n,p}^4\frac{k^2(\log n)^{4/\alpha}(k\log(ep/k))^{2/T_1(\alpha/2)}}{n^2}\right).
\end{align*}
\item[(a)] under Assumption~\ref{eq:Joint},
\begin{align*}
\sup_{M\in\mathcal{M}(k)}&\norm{\hat{\beta}_{n,M} - \beta_{n,M}}_2\\
&= O_p\left(\sqrt{\frac{\Upsilon_{n,k}^{\Gamma}k\log(ep/k)}{n}} + K_{n,p}^2\frac{(\log n)^{2/\alpha}(k\log(ep/k))^{1/T_1(\alpha/2)}}{n}\right),
\end{align*}
and
\begin{align*}
&\sup_{M\in\mathcal{M}(k)}\norm{\hat{\beta}_{n,M} - \beta_{n,M} - \frac{1}{n}\sum_{i=1}^n \left[\Sigma_n(M)\right]^{-1}X_i(M)\left(Y_i - X_i^{\top}(M)\beta_{n,M}\right)}_2\\
&\quad= O_p\left(\frac{\max\{\Upsilon_{n,k}^{\Gamma}, \Upsilon_{n,k}^{\Sigma}\}k\log(ep/k)}{n} + K_{n,p}^4\frac{(\log n)^{4/\alpha}(k\log(ep/k))^{2/T_1(\alpha/2)}}{n^2}\right).
\end{align*}
\end{enumerate}
\end{thm}
\begin{rem}\,(Simplified Rates of Convergence)
The result can be made much more precise by giving the exact tail bound for all the quantities using the exact result of Theorem~\ref{thm:ExactTailIndependence}. We leave the details to the reader. From Theorem~\ref{thm:LinearRepIndependence}, it is clear that if $k\log(ep/k)^{2/T_1(\alpha)} = o(n)$, then the least squares linear regression estimator is uniformly consistent at the rate of $\sqrt{k\log(ep/k)/n}$, which is well-known to be the minimax optimal rate of convergence for high-dimensional linear regression estimators under a true linear model with a sparse parameter vector. We conjecture these rates to be optimal. However, we have not derived minimax rates for this problem. Also, our results are uniform over all probability distributions of the random vectors $(X_i, Y_i)$ satisfying either of the Assumptions~\ref{eq:Marginal} or~\ref{eq:Joint} with $K_{n,p} \le K$ for some fixed constant $K < \infty$.
\end{rem}
\begin{rem}\,(Fixed Covariates)
The results in this section {do} not require any special properties of the data generating distribution such as linearity and {Gaussianity}. The results only require independence of random vectors with weak exponential tails, but it is \emph{not} assumed that $(X_i, Y_i)$ have identical distributions for $1 \le i\le n$. 

It is worth mentioning a special case of our setting that is popular in the classical as well as modern linear regression literature: the setting of fixed covariates. As explained in \cite{Buja14}, this assumption has its roots in the ancillarity theory assuming the truth of a linear model. If the covariates are non-stochastic, then
\[
\hat{\Sigma}_n = \frac{1}{n}\sum_{i=1}^n X_iX_i^{\top} = \frac{1}{n}\sum_{i=1}^n \mathbb{E}\left[X_iX_i^{\top}\right] = \Sigma_n,
\]
so that $\RIP_n(k) = 0$ for all $n$ and $k$. Therefore, the bounds in Theorem~\ref{thm:SimplifiedLinearRep} become trivial in the sense that the uniform linear representation error becomes zero. The result applies because assumption~\ref{eq:Marginal} holds with
\[
K_{n,p} = \max\{\max_{1\le i\le n}\norm{X_i}_{\infty}, \max_{1\le i\le n}\norm{Y_i}_{\psi_{\alpha}}\}.
\]
Also, note from Theorem~\ref{thm:L2UniformConsis} that
\[
\sup_{M\in\mathcal{M}(k)}\norm{\hat{\beta}_{n,M} - \beta_{n,M}}_2 \le \frac{\mathcal{D}_n(k)}{\Lambda_n(k)},
\]
which again leads to the same rate of convergence $\sqrt{k\log(ep/k)/n}$. An interesting observation here is that there is no dependence on the strength of linear association $S_{2,k}(\Sigma_n, \Gamma_n)$ defined in Equation~\eqref{eq:StrengthS} in the case of fixed covariates.
\end{rem}
\begin{rem}{\,(Are the rates optimal?)} 
We believe the rates for the uniform linear representation error to be optimal; cf. Theorem 5.1 of~\cite{javanmard2018debiasing}. An intuitive reason is as follows. Any symmetric function of independent random variables can be expanded as a sum of degenerate $U$-statistics of increasing order according to the Hoeffding decomposition; see \cite{vanZwet84}. That is,
\[
f(W_1, \ldots, W_n) = \mathcal{U}_{1n} + \mathcal{U}_{2n} + \ldots + \mathcal{U}_{nn},
\]
for any symmetric function $f$ of independent random variables $W_1, \ldots, W_n$. Here $\mathcal{U}_{in}$ represents an $i$-th order degenerate $U$-statistics.

For the statistic $\hat{\beta}_{n,M} - \beta_{n,M}$, the first order term $\mathcal{U}_{1n}$ in the decomposition is given by
\[
\mathcal{U}_{1n}^{(M)} = \frac{1}{n}\sum_{i=1}^n \left[\Sigma_n(M)\right]^{-1}X_i(M)\left(Y_i - X_i^{\top}(M)\beta_{n,M}\right).
\]
Hence the difference $\hat{\beta}_{n,M} - \beta_{n,M} - \mathcal{U}_{1n}^{(M)}$ is of the same order as the second order $U$-statistics $\mathcal{U}_{2n}^{(M)}$ next in the decomposition. It is well-known that under mild conditions, a second order degenerate $U$-statistics is of order $1/n$; see \citet[Chapter 5]{Serfling80} for precise results. Therefore, bounding the supremum of the $\norm{\cdot}_2$-norm in the uniform linear representation by
\[
2\max_{|M| \le k}\,\max_{\theta\in\mathbb{R}^{|M|}, \norm{\theta}_2 \le 1}\,\theta^{\top}\left(\hat{\beta}_{n,M} - \beta_{n,M} - \mathcal{U}_{1n}^{(M)}\right)\quad\approx\quad 2\max_{|M| \le k}\,\max_{\theta\in\mathbb{R}^{|M|}, \norm{\theta}_2 \le 1}\,\theta^{\top}\mathcal{U}_{2n}^{(M)},
\]
we see that this is a maximum of at most $(5ep/k)^k$ many degenerate $U$-statistics of order 2, which is expected to be of order $(\log(5ep/k)^k)/n = (k\log(5ep/k))/n$. See \cite{DeLaPena99} for results about suprema of degenerate $U$-statistics.
\end{rem}
\begin{rem}\,(Using covariance matrices instead of Gram matrices)\label{rem:GramCov}
The quantities $\Upsilon_{n,k}^{\Gamma}$ and $\Upsilon_{n,k}^{\Sigma}$ play an important role in determining the exact rates of convergence in Theorem~\ref{thm:LinearRepIndependence}. Under Assumption~\ref{eq:Joint}, it can be easily shown that these quantities are of the same order as $K_{n,p}.$ In cases where the dimension grows, Assumption~\ref{eq:Joint} cannot be justified with non-zero mean of $X_i$'s unless $\norm{\mathbb{E}[X_i]}_2 = O(1)$. Under Assumption~\ref{eq:Marginal}, $\Upsilon_{n,k}^{\Gamma}$ and $\Upsilon_{n,k}^{\Sigma}$ can grow with $k$, and it is hard to pinpoint their growth rate. 
In many cases, it is reasonable to assume a bounded operator norm of the covariance matrix instead of the second moment (or Gram) matrix. For this reason, it is of interest to analyze the least squares estimators with centered random vectors. In this case $\hat{\Sigma}_n$ and $\hat{\Gamma}_n$ should be replaced by
\[
\hat{\Sigma}_n^* := \frac{1}{n}\sum_{i=1}^n \left(X_i - \bar{X}\right)\left(X_i - \bar{X}\right)^{\top}\quad\mbox{and}\quad \hat{\Gamma}_n^* := \frac{1}{n}\sum_{i=1}^n \left(X_i - \bar{X}\right)\left(Y_i - \bar{Y}\right).
\]
Here $\bar{X}$ and $\bar{Y}$ represent the sample means of the covariates and the response, respectively. Without the assumption of equality of $\mathbb{E}[X_i]$ for $1\le i\le n$, $\hat{\Sigma}_n^*$ is not consistent for the covariance matrix of $\bar{X}$. Define
\[
\bar{\mu}_n^{X} := \frac{1}{n}\sum_{i=1}^n \mathbb{E}\left[X_i\right]\quad\mbox{and}\quad \bar{\mu}_n^Y := \frac{1}{n}\sum_{i=1}^n \mathbb{E}\left[Y_i\right].
\]
It is easy to prove that
\begin{align*}
\hat{\Sigma}_n^* &= \frac{1}{n}\sum_{i=1}^n \left(X_i - \bar{\mu}_n^X\right)\left(X_i - \bar{\mu}_n^X\right)^{\top} - \left(\bar{X}_n - \bar{\mu}_n^X\right)\left(\bar{X}_n - \bar{\mu}_n^X\right)^{\top}\\
&= \tilde{\Sigma}_n - \left(\bar{X}_n - \bar{\mu}_n^X\right)\left(\bar{X}_n - \bar{\mu}_n^X\right)^{\top},
\end{align*}
where
\[
\tilde{\Sigma}_n := \frac{1}{n}\sum_{i=1}^n \left(X_i - \bar{\mu}_n^X\right)\left(X_i - \bar{\mu}_n^X\right)^{\top}.
\]
Similarly, we get
\[
\hat{\Gamma}_n^* = \tilde{\Gamma}_n - \left(\bar{X} - \bar{\mu}_n^X\right)\left(\bar{Y} - \bar{\mu}_n^Y\right),\quad\mbox{where}\quad \tilde{\Gamma}_n := \frac{1}{n}\sum_{i=1}^n \left(X_i - \bar{\mu}_n^X\right)\left(Y_i - \bar{\mu}_n^Y\right).
\]
Note that $\tilde{\Gamma}_n$ and $\tilde{\Sigma}_n$ are averages of independent random vectors and random matrices and so the theory before applies with the target vector and matrix given by
\[
\Gamma_n^* = \frac{1}{n}\sum_{i=1}^n \mathbb{E}\left[\left(X_i - \bar{\mu}_n^X\right)\left(Y_i - \bar{\mu}_n^Y\right)\right]\quad\mbox{and}\quad \Sigma_n^* = \frac{1}{n}\sum_{i=1}^n \mathbb{E}\left[\left(X_i - \bar{\mu}_n^X\right)\left(X_i - \bar{\mu}_n^X\right)^{\top}\right].
\]
It is important to recognize that Theorem~\ref{thm:SimplifiedLinearRep} is not directly applicable since the forms of $\hat{\Sigma}_n^*$ and $\hat{\Gamma}_n^*$ do not match the structure required. One has to apply Theorem~\ref{thm:UniformLinearRep} to obtain
\begin{align*}
\sup_{M\in\mathcal{M}(k)}&\norm{\hat{\beta}_{M}^* - \beta_{M}^* - \frac{1}{n}\sum_{i=1}^n \left[\Sigma_n^*(M)\right]^{-1}\left(X_i - \bar{\mu}_n^X\right)(M)\left\{Y_i - \bar{\mu}_n^Y - (X_i(M) - \bar{\mu}_n^X(M))^{\top}\beta_{M}^*\right\}}_2\\
&\le \frac{\mathcal{D}(k, \bar{X} - \bar{\mu}_n^X)\left[|\bar{Y} - \bar{\mu}_n^Y| + \mathcal{D}(k, \bar{X} - \bar{\mu}_n^X)S_{2,k}^*\right]}{\Lambda_n^*(k)} + \frac{\RIP_n^*(k)}{\Lambda_n^*(k)}\times\frac{\mathcal{D}_n^*(k) + \RIP_n^*(k)S_{2,k}^*}{\Lambda_n^*(k) - \RIP_n^*(k)},
\end{align*}
where
\[
\hat{\beta}_{M}^* := \beta_M(\hat{\Sigma}_n^*, \hat{\Gamma}_n^*),\quad \beta_{M}^* := \beta_M(\Sigma_n^*, \Gamma_n^*),\quad S_{2,k}^* := S_{2,k}(\Sigma_n^*,\Gamma_n^*).
\]
and
\[
\RIP_n^*(k) := \RIP(k, \hat{\Sigma}_n^* - \Sigma_n^*),\quad\mathcal{D}_n^*(k) := \mathcal{D}(k,\hat{\Gamma}_n^* - \Gamma_n^*),\quad\mbox{and}\quad \Lambda_n^*(k) := \Lambda(k; \Sigma_n^*).
\]
From the calculations presented above, it follows that
\begin{align*}
\RIP_n^*(k) &\le \RIP(k, \tilde{\Sigma}_n - \Sigma_n^*) + \mathcal{D}^2(k, \bar{X} - \bar{\mu}_n^X),\\
\mathcal{D}_n^*(k) &\le \mathcal{D}(k, \tilde{\Gamma}_n - \Gamma_n^*) + \mathcal{D}(k, \bar{X} - \bar{\mu}_n^X)\left|\bar{Y} - \mu_n^Y\right|.
\end{align*}
The right hand side terms above can be controlled using Theorem~\ref{thm:IndependenceTailBound}. Thus, the linear representation changes when using the sample covariance matrix. See Section 4.1.1 of~\cite{KuchAbhi17} for more details.
\end{rem}
\section{Rates for Functionally Dependent Observations}\label{sec:Dependent}
In this section, we extend all the results presented in the previous section to dependent data. The dependence structure on the observations we use is based on a notion developed by \cite{Wu05}. It is possible to derive these results also under the classical dependence notions like $\alpha$-,$\beta$-,$\rho$- mixing, however, verifying the mixing assumptions can often be hard and many well-known processes do not satisfy them. See \cite{Wu05} for more details. It has also been shown that many econometric time series can be studied under the notion of functional dependence; see \cite{Wu10}, \cite{Liu13} and \cite{WuWu16}. For a study of dependent processes under a similar framework called $L_p$-approximability, see \cite{Potscher97}.

The dependence notion of \cite{Wu05} is written in terms of an input-output process that is easy to analyze in many settings. The process is defined as follows. Let $\{\varepsilon_i, \varepsilon_i':\,i\in\mathbb{Z}\}$ denote a sequence of independent and identically distributed random variables on some measurable space $(\mathcal{E}, \mathcal{B})$. 
Define the $q$-dimensional process $W_i$ with causal representation as
\begin{equation}\label{eq:CausalProcessVector}
W_i = G_i(\ldots,\varepsilon_{i-1}, \varepsilon_i)\in\mathbb{R}^q,
\end{equation}
for some vector-valued function $G_i(\cdot) = (g_{i1}(\cdot), \ldots, g_{iq}(\cdot))$. 
By Wold representation theorem for stationary processes, this causal representation holds in many cases. Define the non-decreasing filtration
\begin{equation}\label{eq:Filtration}
\mathcal{F}_i := \sigma\left(\ldots, \varepsilon_{i-1}, \varepsilon_i\right).
\end{equation}
Using this filtration, we also use the notation $W_i = G_i(\mathcal{F}_i)$. To measure the strength of dependence, define for $r\ge 1$ and $1\le j\le q$, the \textbf{functional dependence measure}
\begin{equation}\label{eq:FunctionalDepMeasure}
\delta_{s,r,j} := \max_{1\le i\le n}\,\norm{W_i(j) - W_{i,s}(j)}_r,\quad\mbox{and}\quad \Delta_{m,r,j} := \sum_{s = m}^{\infty} \delta_{s,r,j},
\end{equation}
where 
\begin{equation}\label{eq:ZikDef}
W_{i,s}(j) := g_{ij}(\mathcal{F}_{i,i-s})\quad\mbox{with}\quad \mathcal{F}_{i,i-s} := \sigma\left(\ldots,\varepsilon_{i-s-1}, \varepsilon_{i-s}', \varepsilon_{i-s+1}, \ldots, \varepsilon_{i-1}, \varepsilon_i\right).
\end{equation}
The $\sigma$-field $\mathcal{F}_{i,i-s}$ represents a coupled version of $\mathcal{F}_i$. The quantity $\delta_{s,r,j}$ measures the dependence using the distance in terms of $\norm{\cdot}_r$-norm between $g_{ij}(\mathcal{F}_i)$ and $g_{ij}(\mathcal{F}_{i,i-s}).$ In other words, it is quantifying the impact of changing $\varepsilon_{i-s}$ on $g_{ij}(\mathcal{F}_i)$; see Definition 1 of \cite{Wu05}. The \textbf{dependence adjusted norm} for the $j$-th coordinate is given by
\[
\norm{\{W(j)\}}_{r,\nu} := \sup_{m\ge 0}\, (m + 1)^{\nu}\Delta_{m,r,j},\quad \nu \ge 0. 
\]
To summarize these measures for the vector-valued process, define
\[
\norm{\{W\}}_{r,\nu} := \max_{1\le j\le q}\,\norm{\{W(j)\}}_{r,\nu}\quad\mbox{and}\quad \norm{\{W\}}_{\psi_{\alpha}, \nu} := \sup_{r\ge 2}\,r^{-1/\alpha}\norm{\{W\}}_{r,\nu}.
\]
\begin{rem}\,(Independent Sequences)\label{rem:Indep}
Any notion of dependence should at least include independent random variables. It might be helpful to understand how independent random variables fits into this framework of dependence. For independent random vectors $W_i$, the causal representation reduces to
\[
W_i = G_i(\ldots, \varepsilon_{i-1}, \varepsilon_i) = G_i(\varepsilon_i)\in\mathbb{R}^{q}.
\]
It is not a function of any of the previous $\varepsilon_j, j < i$. This implies by the definition~\eqref{eq:ZikDef}  that 
\[
W_{i,s} = \begin{cases}
G_i(\varepsilon_i) = W_i,&\mbox{if }s\ge 1,\\
G_i(\varepsilon_{i}') =: W_i',&\mbox{if }s = 0.
\end{cases}
\] 
Here $W_i'$ represents an independent and identically distributed copy of $W_i$. Hence, 
\[
\delta_{s,r,j} = \begin{cases}
0, &\mbox{if }s\ge 1,\\
\norm{W_i(j) - W_i'(j)}_r \le 2\norm{W_i(j)}_r,&\mbox{if }s = 0.
\end{cases}
\]
It is now clear that for any $\nu > 0$,
\[
\norm{\{W\}}_{r, \nu} = \sup_{m\ge 0}\,(m + 1)^{\nu}\Delta_{m,r} = \Delta_{0, r} \le 2\max_{1\le j\le q}\norm{W_i(j)}_r.
\]
Hence, if the independent sequence $W_i$ satisfies assumption~\ref{eq:Marginal}, then $\norm{\{W\}}_{\psi_{\alpha}, \nu} < \infty$ for all $\nu > 0$, in particular for $\nu = \infty$. Therefore, independence corresponds to $\nu = \infty$. As $\nu$ decreases to zero, the random vectors become more and more dependent.
\end{rem}
All our results in this section are based on the following tail bound for the maximum of averages of functionally dependent variables which is an extension of Theorem 2 of \cite{WuWu16}. This result is similar to Theorem~\ref{thm:IndependenceTailBound}. For this result, define
\begin{equation}\label{eq:ShiftAndTrucMainPaper}
s(\lambda) := (1/2 + 1/\lambda)^{-1},\quad\mbox{and}\quad T_1(\lambda) := \min\{\lambda, 1\}\quad\mbox{for all}\quad \lambda > 0. 
\end{equation}
\begin{thm}\label{thm:DependentSums}
Suppose $Z_1, \ldots, Z_n$ are random vectors in $\mathbb{R}^q$ with a causal representation such as \eqref{eq:CausalProcessVector} with mean zero. Assume that for some $\alpha > 0$, and $\nu > 0$,
\begin{equation}\label{eq:AssumptionDependenceMain}
\norm{\{Z\}}_{\psi_{\alpha}, \nu} = \sup_{r\ge 2}\,\sup_{m\ge 0}\, r^{-1/\alpha}(m+ 1)^{\nu}\Delta_{m,r} \le K_{n,q}.
\end{equation}
Define
\[
\Omega_n(\nu) := 2^{\nu}\times\begin{cases}
5/(\nu - 1/2)^3, &\mbox{if }\nu > 1/2,\\
2(\log_2n)^{5/2}, &\mbox{if }\nu = 1/2,\\
5(2n)^{(1/2 - \nu)}/(1/2 - \nu)^3,&\mbox{if }\nu < 1/2.
\end{cases}
\]
Then for all $t\ge 0$, with probability at least $1 - 8e^{-t}$,
\begin{equation}\label{eq:TailBoundMain}
\begin{split}
\max_{1\le j\le q}\left|\sum_{i=1}^n Z_i(j)\right| &\le e\sqrt{n}\norm{\{Z\}}_{2,\nu}B_{\nu}\sqrt{t + \log(q+1)}\\ &\qquad+ C_{\alpha}K_{n,q}(\log n)^{1/s(\alpha)}\Omega_n(\nu)(t + \log(q + 1))^{1/T_1(s(\alpha))}.
\end{split}
\end{equation}
Here $B_{\nu}$ and $C_{\alpha}$ are constants depending only on $\nu$ and $\alpha$, respectively.
\end{thm}
\begin{proof}
The proof follows from Theorem~\ref{thm:DependentSumsApp} proved in Appendix~\ref{AppSec:Dependence} and a union bound.
\end{proof}

Getting back to the application of uniform-in-submodel results for linear regression, we assume that the random vectors are elements of a causal process with exponential tails. Formally, suppose $(X_1, Y_1), \ldots, (X_n, Y_n)$ 
are random vectors in $\mathbb{R}^p\times\mathbb{R}$ satisfying the following assumption:
\begin{description}
\item[\namedlabel{eq:Dependent}{(DEP)}] Assume that there exist $n$ vector-valued functions $G_i$ and an iid sequence $\{\varepsilon_i:\,i\in\mathbb{Z}\}$ such that
\[
W_i := (X_i, Y_i) = G_i(\ldots, \varepsilon_{i-1}, \varepsilon_i)\in\mathbb{R}^{p+1}.
\]
Also, for some $\nu, \alpha > 0$,
\[
\norm{\{W\}}_{\psi_{\alpha}, \nu} \le K_{n,p}\quad\mbox{and}\quad \max_{1\le i\le n}\max_{1\le j\le p+1}\,|\mathbb{E}\left[W_i(j)\right]| \le K_{n,p}.
\]
\end{description}
Based on Remark~\ref{rem:Indep}, Assumption~\ref{eq:Dependent} is equivalent to Assumption~\ref{eq:Marginal} for independent data. For independent random variables, the second part of Assumption~\ref{eq:Dependent} about the expectations follows from the $\psi_{\alpha}$-bound assumption. The reason for this expectation bound in the assumption here is that the functional dependence measure $\delta_{s,r}$ does not have any information about the expectation since
\[
\norm{W_i(j) - W_{i,s}(j)}_r = \norm{\left(W_i(j) - \mathbb{E}\left[W_i(j)\right]\right) - \left(W_{i,s}(j) - \mathbb{E}\left[W_{i,s}(j)\right]\right)}_r.
\]
The coupled random variable $W_{i,s}$ has the same expectation as $W_i$. Since the quantities we need to bound involve products of random variables, such a bound on the expectations is needed for our analysis.

We are now ready to state the final results of this section. Only results similar to Theorems~\ref{thm:ExactTailIndependence} and~\ref{thm:LinearRepIndependence} are stated. Also, we only state the results under marginal moment assumption and the version with joint moment assumption can easily be derived based on the proof. These results are based on Theorem~\ref{thm:DependentSums}. Recall from inequalities~\eqref{eq:FiniteDnk} and~\eqref{eq:FiniteRIPnk} that
\begin{align*}
\mathcal{D}_n(k) &\le 2\sup_{\theta\in\mathcal{N}(1/2,\Theta_k)}\,\left|\frac{1}{n}\sum_{i=1}^n \left\{\theta^{\top}X_iY_i - \mathbb{E}\left[\theta^{\top}X_iY_i\right]\right\}\right|,\\
\RIP_n(k) &\le 2\sup_{\theta\in\mathcal{N}(1/4,\Theta_k)}\,\left|\frac{1}{n}\sum_{i=1}^n \left(X_i^{\top}\theta\right)^2 - \mathbb{E}\left[\left(X_i^{\top}\theta\right)^2\right]\right|.
\end{align*}
Note that these quantities involve linear combinations $(\theta^{\top}X_i)$ and products $(\theta^{\top}X_iY_i)$ of functionally dependent random variables. It is clear that all linear combinations and products of functionally dependent random variables have a causal representation since if $W_i^{(1)} := h_i^{(1)}(\mathcal{F}_i)$ and $W_i^{(2)} := h_i^{(2)}(\mathcal{F}_i)$, then
\begin{align*}
\alpha W_i^{(1)} + \beta W_i^{(2)} = \alpha h_i^{(1)}(\mathcal{F}_i) + \beta h_i^{(2)}(\mathcal{F}_i)\quad&\mbox{and}\quad W_i^{(1)}W_i^{(2)} = h_i^{(1)}(\mathcal{F}_i)h_i^{(2)}(\mathcal{F}_i).
\end{align*}
Thus, they can be studied under the same framework of dependence. In Lemma~\ref{lem:ProductProcess}, we bound the functional dependence measure of such linear combination and product processes.

For the main results of this section, define for $\theta\in\Theta_k$ (see~\eqref{eq:kSparseSubsetDef})
\begin{align*}
\vartheta_4^{(\Gamma)}(\theta) &:= \left(\norm{\{\theta^{\top}X\}}_{4,0} + \max_{1\le i\le n}\left|\mathbb{E}\left[\theta^{\top}X_i\right]\right|\right)\norm{\{Y\}}_{4,\nu}\\ &\qquad+ \left(\norm{\{Y\}}_{4,0} + \max_{1\le i\le n}\left|\mathbb{E}\left[Y_i\right]\right|\right)\norm{\{\theta^{\top}X\}}_{4,\nu},\\
\vartheta_4^{(\Sigma)}(\theta) &:= 2\left(\norm{\{\theta^{\top}X\}}_{4,0} + \max_{1\le i\le n}\left|\mathbb{E}\left[\theta^{\top}X_i\right]\right|\right)\norm{\{\theta^{\top}X\}}_{4,\nu}.
\end{align*}
\begin{thm}\label{thm:ExactTailDependence}
Fix $n, k\ge 1$ and let $t\ge 0$ be any real number. Define
\begin{align*}
\sqrt{\Upsilon_{n,k}^{\Gamma}} &:= \sup_{\theta\in\Theta_k}\,\vartheta_4^{(\Gamma)}(\theta),\quad\mbox{and}\quad\sqrt{\Upsilon_{n,k}^{\Sigma}} := \sup_{\theta\in\Theta_k}\,\vartheta_4^{(\Sigma)}(\theta).
\end{align*}
Then under Assumption~\ref{eq:Dependent}, with probability at least $1 - 16e^{-t}$, the following inequalities hold simultaneously,
\begin{align*}
\mathcal{D}_{n}(k) &\le 2eB_{\nu}\sqrt{\frac{\Upsilon_{n,k}^{\Gamma}(t + k\log(3ep/k))}{n}}\\ 
&\qquad+ C_{\alpha}K_{n,p}^2\frac{k^{1/2}(\log n)^{1/s(\alpha/2)}\Omega_n(\nu)(t + k\log(3ep/k))^{1/{T_1(s(\alpha/2))}}}{n},
\end{align*}
and
\begin{align*}
\mathcal{\RIP}_n(k) &\le 2eB_{\nu}\sqrt{\frac{\Upsilon_{n,k}^{\Gamma}(t + k\log(5ep/k))}{n}}\\ 
&\qquad+ C_{\alpha}K_{n,p}^2\frac{k(\log n)^{1/s(\alpha/2)}\Omega_n(\nu)(t + k\log(5ep/k))^{1/{T_1(s(\alpha/2))}}}{n}.
\end{align*}
Here $T_1(\alpha)$ and $s(\alpha)$ are functions given in~\eqref{eq:ShiftAndTrucMainPaper} and $B_{\nu}, C_{\alpha}$ are constants depending only on $\nu$ and $\alpha$, respectively.
\end{thm}
\begin{proof}
By Lemma~\ref{lem:ProductProcess} and Assumption~\ref{eq:Dependent}, it holds that for all $\theta\in\Theta_k$,
\begin{align*}
\norm{\{\theta^{\top}XY\}}_{2, \nu} &\le \vartheta_4^{\Gamma}(\theta)\quad\mbox{and}\quad \norm{\{(\theta^{\top}X)^2\}}_{2,\nu} \le \vartheta_4^{\Sigma}(\theta).
\end{align*}
Also, using Lemma~\ref{lem:LinearCombinations} and Lemma~\ref{lem:ProductProcess}, it follows that
\begin{align*}
\sup_{r\ge 2}\,r^{-2/\alpha}\norm{\{\theta^{\top}XY\}}_{r,\nu} &\le 3k^{1/2}K_{n,p}^22^{1/\alpha},\\
\sup_{r\ge 2}\,r^{-2/\alpha}\norm{\{(\theta^{\top}X)^2\}}_{r, \nu} &\le 3kK_{n,p}^22^{1/\alpha}.
\end{align*}
Hence applying Theorem~\ref{thm:DependentSums}, the result is proved.
\end{proof}
Theorem~\ref{thm:ExactTailDependence} along with Theorem~\ref{thm:SimplifiedLinearRep} implies the following uniform linear representation result for linear regression under functional dependence. Recall the notation $\Lambda_n(k)$ from Equation~\eqref{eq:SimplifiedNotation} and also $\hat{\beta}_{n,M}, \beta_{n,M}$ from Equations~\eqref{eq:LeastSquaresEstimatorClosed}, \eqref{eq:LeastSquaresExpected}.
\begin{thm}\label{thm:LinearRepDependence}
If $(\Lambda_n(k))^{-1} = O(1)$ as $n, p\to\infty$, then under Assumption~\ref{eq:Dependent}, the following rates of convergence hold as $n\to\infty.$  
\begin{align*}
&\sup_{M\in\mathcal{M}(k)}\norm{\hat{\beta}_{n,M} - \beta_{n,M}}_2\\
&\quad= O_p\left(\sqrt{\frac{\Upsilon_{n,k}^{\Gamma}k\log(ep/k)}{n}} + K_{n,p}^2\frac{k^{1/2}(\log n)^{1/s(\alpha/2)}\Omega_n(\nu)(k\log(ep/k))^{1/{T_1(s(\alpha/2))}}}{n}\right),
\end{align*}
and
\begin{align*}
&\sup_{M\in\mathcal{M}(k)}\norm{\hat{\beta}_{n,M} - \beta_{n,M} - \frac{1}{n}\sum_{i=1}^n \left[\Sigma_n(M)\right]^{-1}X_i(M)\left(Y_i - X_i^{\top}(M)\beta_{n,M}\right)}_2\\
&\quad= O_p\left(\frac{\max\{\Upsilon_{n,k}^{\Gamma}, \Upsilon_{n,k}^{\Sigma}\}k\log(ep/k)}{n}\right)\\
&\quad\qquad + K_{n,p}^4O_p\left(\frac{k^2(\log n)^{2/s(\alpha/2)}(k\log(ep/k))^{2/T_1(s(\alpha/2))}\Omega_n^2(\nu)}{n^2}\right).
\end{align*}
\end{thm}
In comparison to Theorem~\ref{thm:LinearRepIndependence}, the rates attained here are very similar except for two changes:
\begin{enumerate}
\item The exponent terms $\alpha/2,$ and $T_1(\alpha/2))$ are replaced by $s(\alpha/2),$ and $T_1(s(\alpha/2))$, respectively. This is because of the use of a version of Burkholder's inequality from \cite{Rio09} in the proof of Theorem~\ref{thm:DependentSumsApp}.
\item The factor $\Omega_n(\nu)$ in the second order terms above. This factor is due to the dependence of the process. If $\nu > 1/2$ (which corresponds to ``weak'' dependence), then $\Omega_n(\nu)$ is of order 1 and for the boundary case $\nu = 1/2$, $\Omega_n(\nu)$ is of order $(\log n)^{5/2}.$ In both these cases, the rates obtained for functionally dependent $\psi_{\alpha}$-random vectors match very closely the rates obtained for independent $\psi_{s(\alpha)}$-random vectors.
\end{enumerate}
\begin{rem}\;(Some Comments on Assumption~\ref{eq:Dependent})
Assumption~\ref{eq:Dependent} is similar to the one used in Theorem 3.3 of \cite{ZhangWu17} for derivation of a high-dimensional central limit theorem with logarithmic dependence on the dimensional $p$. It is worth mentioning that in their notation $\alpha$ corresponds to the functional dependence and $\nu$ corresponds to the moment assumption. Also their assumption is written as
\[
\sup_{r\ge 2}\,\frac{\norm{\{Z\}}_{r,\nu}}{r^{\alpha}} < \infty,\quad \mbox{(after swapping the dependence and moment parameters)}.
\]
Our assumption, however, is written as
\[
\sup_{r\ge 2}\,\frac{\norm{\{Z\}}_{r,\nu}}{r^{1/\alpha}} < \infty.
\]
Hence, our parameters $(\alpha, \nu)$ correspond to their parameters $(1/\nu, \alpha)$. Our assumptions are weaker than those used by \cite{Zhang14}. From the discussion surrounding Equation (28) there, they require geometric decay of $\Delta_{m,r,j}$ while we only require polynomial decay. \cite{ZhangWu17} only deal with stationary sequences and \cite{Zhang14} allows non-stationarity. Some useful examples verifying the bounds on the functional dependence measure are also provided in \cite{Zhang14}.
\end{rem}
\section{Discussion and Conclusions}\label{sec:Discussion}
In this paper, we have proved uniform-in-submodel results for the least squares linear regression estimator under a model-free framework allowing for the total number of covariates to diverge ``almost exponentially'' in $n$. Our results are based on deterministic inequalities. The exact rate bounds are provided when the random vectors are independent and functionally dependent. In both cases, the random variables are assumed to have weak exponential tails to provide logarithmic dependence on the dimension $p$. 

In this paper, we have primarily focused on ordinary least squares linear regression. The main results, uniform-in-submodels consistency and linear representation, continue to hold for a large class of $M$-estimators defined by twice differentiable loss function as shown in~\cite{kuchibhotla2018Deter}. The implications of these results are that one can use all the information from all the observations to build a submodel (subset of covariates) and apply a general $M$-estimation technique on the final model selected. These results can be extended to non-differentiable loss functions using techniques from empirical process theory, in particular, the stochastic uniform equicontinuity assumption. See, for example,~\citet[Chapter 2]{giessing2018high} for results under independence.

All of our results are free of the assumption of correctly specified models. Therefore, our results provide a ``target'' $\beta_{n,M}$ for the estimator $\hat{\beta}_{n,M}$ irrespective of whether $M$ is fixed or random as long as $|M| \le k$. This implication follows from the uniform-in-submodel feature of the results. The conclusion here is that if the statistician has a target in mind, then all they need to check is if $\beta_{n,M}$ is close to the target they are thinking of.

As mentioned in the beginning of the article one can rethink high-dimensional linear regression as using high-dimensional data for exploration to find a ``significant'' set of variables and then applying the ``low-dimensional'' linear regression technique. If the exploration is not restricted to a very principled method, then inference can be very difficult. This problem is exactly equivalent to the problem of valid post-selection inference. {Post-selection inference has a rich history in both statistics and econometrics. \cite{leeb2005model,Leeb06,leeb2006performance,leeb2008can} have provided several impossibility results regarding the estimation of the distribution of $\widehat{\beta}_{\widehat{M}}$, when $\widehat{M}$ represents the data-dependent selected model. One way to avoid this difficulty is by performing inference for all models simultaneously.} The results in this paper allow for the construction of a simultaneous inference procedure using a high-dimensional central limit theorem and multiplier bootstrap; see~\cite{bachoc2019valid,Bac16,kuchibhotla2021high} and~\citet[Section 2]{belloni2018high} for more details. A related exploration will be provided in a future manuscript.
\appendix
\begin{center}
APPENDIX
\end{center}
\section{Auxiliary Results for Independent Random Vectors}\label{AppSec:Independence}
The following result proves a tail bound for a maximum of the average of mean zero random variables and follows from Theorem 4 of \cite{Adam08}. The result there is only stated for $\alpha\in(0, 1]$, however, the proof can be extended to the case $\alpha > 1$. See the forthcoming paper \cite{KuchAbhi17} for a clear exposition.
\begin{thm}\label{thm:IndependenceTailBound}
Suppose $W_1, \ldots, W_n$ are mean zero independent random vectors in $\mathbb{R}^q, q\ge 1$ such that for some $\alpha > 0$ and $K_{n,q} > 0$,
\[
\max_{1\le i\le n}\max_{1\le j\le q}\,\norm{W_i(j)}_{\psi_{\alpha}} \le K_{n,q}.
\]
Define 
\[
\Gamma_{n,q} := \max_{1\le j\le q}\,\frac{1}{n}\sum_{i=1}^n \mathbb{E}\left[W_i^2(j)\right].
\]
Then for any $t\ge 0,$ with probability at least $1 - 3e^{-t}$,
\begin{align*}
\max_{1\le j\le q}\,\left|\frac{1}{n}\sum_{i=1}^n W_i(j)\right| \le 7\sqrt{\frac{\Gamma_{n,q}(t + \log(2q))}{n}} + \frac{C_{\alpha}K_{n,q}(\log(2n))^{{1}/{\alpha}}(t + \log(2q))^{{1}/{T_1(\alpha)}}}{n},
\end{align*}
where $T_1(\alpha) = \min\{\alpha, 1\}$ and $C_{\alpha}$ is a constant depending only on $\alpha$.
\end{thm}
\begin{proof}
Fix $1\le j\le q$ and apply Theorem 4 of \cite{Adam08} with $\mathcal{F} = \{f\}$ where $f(W_i) = W_i(j)$ for $1\le i\le n$. Then applying the union bound the result follows. To extend the result to the case $\alpha > 1$, use Theorem 5 of \cite{Adam08} with $\alpha = 1$ to bound the second part of inequality (8) there. 
\end{proof}
Using Theorem~\ref{thm:IndependenceTailBound}, we get the following results for RIP and $\mathcal{D}$ under independence.
\begin{thm}\label{thm:ExactTailIndependence}
Fix $n, k\ge 1$ and let $t\ge 0$ be any real number. Then the following probability statements hold true:
\begin{enumerate}
\item[(a)] Under Assumption~\ref{eq:Marginal}, with probability at least $1 - 6e^{-t}$, the following two inequalities hold simultaneously,
\begin{align*}
\mathcal{D}_{n}(k) &\le 14\sqrt{\frac{\Upsilon_{n,k}^{\Gamma}(t + k\log(3ep/k))}{n}}\\ 
&\qquad+ C_{\alpha}K_{n,p}^2\frac{k^{1/2}(\log(2n))^{{2}/{\alpha}}(t + k\log(3ep/k))^{1/{T_1(\alpha/2)}}}{n},
\end{align*}
and
\begin{align*}
\mathcal{\RIP}_n(k) &\le 14\sqrt{\frac{\Upsilon_{n,k}^{\Sigma}(t + k\log(5ep/k))}{n}}\\ &\qquad+ C_{\alpha}K_{n,p}^2\frac{k(\log(2n))^{{2}/{\alpha}}(t + k\log(5ep/k))^{1/{T_1(\alpha/2)}}}{n}.
\end{align*}
\item[(b)] Under Assumption~\ref{eq:Joint}, with probability at least $1 - 6e^{-t}$, the following two inequalities hold simultaneously,
\begin{align*}
\mathcal{D}_{n}(k) &\le 14\sqrt{\frac{\Upsilon_{n,k}^{\Gamma}(t + k\log(3ep/k))}{n}}\\ 
&\qquad+ C_{\alpha}K_{n,p}^2\frac{(\log(2n))^{{2}/{\alpha}}(t + k\log(3ep/k))^{{1}/{T_1(\alpha/2)}}}{n},
\end{align*}
and
\begin{align*}
\mathcal{\RIP}_n(k) &\le 14\sqrt{\frac{\Upsilon_{n,k}^{\Sigma}(t + k\log(5ep/k))}{n}}\\ &\qquad+ C_{\alpha}K_{n,p}^2\frac{(\log(2n))^{{2}/{\alpha}}(t + k\log(5ep/k))^{{1}/{T_1(\alpha/2)}}}{n}.
\end{align*}
Here $T_1(\alpha) = \min\{\alpha, 1\}$ and $C_{\alpha}$ is a constant depending only on $\alpha$.
\end{enumerate}
\end{thm}
\begin{proof}
These bounds follow from Theorem~\ref{thm:IndependenceTailBound} and inequalities~\eqref{eq:FiniteDnk} and~\eqref{eq:FiniteRIPnk}. To bound $\mathcal{D}_n(k)$, we take
\[
W_i := (\theta^{\top}X_iY_i)_{\theta\in \mathcal{N}(1/2, \Theta_k)},
\]
in Theorem~\ref{thm:IndependenceTailBound}. Because $|\mathcal{N}(1/2, \Theta_k)| \le (3ep/k)^k$, the result follows. Similarly for $\mathcal{\RIP}_n(k)$, we take
\[
W_i := ((\theta^{\top}X_i)^2)_{\theta\in\mathcal{N}(1/4, \Theta_k)},
\]
in Theorem~\ref{thm:IndependenceTailBound}.
\end{proof}
\section{Auxiliary Results for Dependent Random Vectors}\label{AppSec:Dependence}
In this section, we present a moment bound for sum of functionally dependent mean zero real-valued random variables. The moment bound here is an extension of Theorem 2 of \cite{WuWu16} to random variables with exponential tails. The main distinction is that our moment bound exhibits a part Gaussian behavior. For proving these moment bounds, we need a few preliminary results and notation. Suppose $Z_1 \ldots, Z_n$ are mean zero real valued random variables with a causal representation 
\begin{equation}\label{eq:CausalProcess}
Z_i = g_i(\ldots, \varepsilon_{i-1}, \varepsilon_i),
\end{equation}
for some real valued function $g_i$. We write $\delta_{k,r} = \norm{Z_i - Z_{i,k}}_r$. The following proposition bounds the $r$-th moment of $Z_i$ in terms of $\norm{\{Z\}}_{r,\nu}$. This is based on the calculation shown after Equation (2.8) in \cite{WuWu16}. 
\begin{prop}\label{prop:Converse}
Consider the setting above. If $\mathbb{E}\left[Z_i\right] = 0$ for $1\le i\le n$, then $$\norm{Z_i}_r \le \norm{\{Z\}}_{r,0} \le \norm{\{Z\}}_{r,\nu},\quad\mbox{for any}\quad r\ge 1\quad\mbox{and}\quad \nu > 0.$$
\end{prop}
\begin{proof}
Assuming $\mathbb{E}\left[Z_i\right] = 0$ for $1\le i\le n$, it follows that
\[
Z_i = \sum_{\ell = -\infty}^{i}\left(\mathbb{E}\left[Z_i\big|\mathcal{F}_{\ell}\right] - \mathbb{E}\left[Z_i\big|\mathcal{F}_{\ell - 1}\right]\right),
\]
and so,
\begin{align*}
\norm{Z_i}_r &\le \sum_{\ell = -\infty}^{i} \norm{\mathbb{E}\left[Z_i\big|\mathcal{F}_{\ell}\right] - \mathbb{E}\left[Z_i\big|\mathcal{F}_{\ell - 1}\right]}_r = \sum_{\ell = -\infty}^{i} \norm{\mathbb{E}\left[Z_i - Z_{i,i-\ell}\big|\mathcal{F}_{-\ell}\right]}_r \le \sum_{\ell = 0}^{\infty} \delta_{\ell, r}.
\end{align*}
The last inequality follows from Jensen's inequality and noting that the last bound equals $\Delta_{0,r}$, it follows that $\norm{Z_i}_r \le \Delta_{0,r} = \norm{\{Z\}}_{r,0}.$
\end{proof}
The following lemma provides a bound on the moments of a martingale in terms of the moments of the martingale difference sequence. This result is an improvement over the classical Burkholder's inequality.
\begin{lem}[Theorem 2.1 of \cite{Rio09}]\label{lem:MartingaleBounds}
Let $\{S_n:\,n\ge 0\}$ be a martingale sequence with $S_0 = 0$ adapted with respect to some non-decreasing filtration $\mathcal{F}_n, n\ge 0$. Let $X_k = S_k - S_{k-1}$ denote the corresponding martingale difference sequence. Then for any $p\ge 2$,
\[
\norm{S_n}_p \le \sqrt{p - 1}\left(\sum_{k = 1}^n \norm{X_k}_p^2\right)^{1/2}.
\]
\end{lem}
The following simple calculation is also used in Theorem~\ref{thm:DependentSumsApp}. Define
\[
L:= \left\lfloor\frac{\log n}{\log 2}\right\rfloor\quad\mbox{and}\quad \lambda_{\ell} := \begin{cases}
3\pi^{-2}\ell^{-2},&\mbox{if }1\le \ell \le L/2,\\
3\pi^{-2}(L + 1 - \ell)^{-2},&\mbox{if }L/2 < \ell \le L.
\end{cases}
\]
\begin{lem}\label{lem:SimpleCalculation}
The following inequalities hold true:
\begin{itemize}
\item[(a)] For any $\beta \ge 0$ and $p\ge 2$,
\begin{equation}\label{eq:FirstIneqReq}
\sum_{\ell = 1}^L \frac{1}{\lambda_{\ell}^p2^{p\ell\beta}} \le 2\sum_{\ell = 1}^{L/2} \frac{1}{\lambda_{\ell}^p2^{p\ell\beta}}\le \begin{cases}\left({5}/{\beta^3}\right)^p\left({\pi^2}/{3}\right)^{p+1}, &\mbox{if }\beta > 0,\\
2(\log_2n)^{2p+1}\left({\pi^2}/{3}\right)^{p+1},&\mbox{if }\beta = 0.
\end{cases}
\end{equation}
\item[(b)] For any $\beta > 0$ and $p\ge 2$, 
\begin{equation}\label{eq:SecondIneqReq}
\sum_{\ell = 1}^{L} \frac{2^{p\ell(1/2 - \beta)}}{\lambda_{\ell}^p} \le \left(\frac{\pi^2}{3}\right)^{p+1}
\begin{cases}
(5/(\beta - 1/2)^3)^p,&\mbox{if }\beta > 1/2,\\
2(\log_2n)^{2p+1},&\mbox{if }\beta = 1/2,\\
(2n)^{(1/2 - \beta)p}(5/(1/2 - \beta)^3)^p,&\mbox{if }\beta < 1/2.
\end{cases}
\end{equation}
\end{itemize}
\end{lem}
\begin{proof}
\begin{enumerate}
\item[(a)] Note that for any $\beta > 0$,
\[
\sup_{\ell > 0}\ell^32^{-\ell\beta} = \ell^3\exp(-(\log 2)\ell\beta) \le \left(\frac{3}{e\beta\log 2}\right)^3 \le \frac{5}{\beta^3},
\]
and so,
\begin{equation*}
\begin{split}
\left(\frac{3}{\pi^2}\right)^p\sum_{\ell = 1}^L \frac{1}{\lambda_{\ell}^p2^{p\ell\beta}} &= \sum_{\ell = 1}^{L/2}\left(\frac{\ell^2}{2^{\ell\beta}}\right)^p + \sum_{\ell = L/2 + 1}^L \left(\frac{(L + 1 - \ell)^2}{2^{\ell\beta}}\right)^p\\
&\le \sum_{\ell = 1}^{L/2} \left(\frac{\ell^2}{2^{\ell\beta}}\right)^p + 2^{-p\beta}\sum_{\ell = 1}^{L/2} \left(\frac{\ell^2}{2^{\ell\beta}}\right)^p\\ & \le 2\left(\frac{5}{\beta^3}\right)^p\sum_{\ell = 1}^{L/2} \frac{1}{\ell^p}\le \frac{\pi^2}{3}\left(\frac{5}{\beta^3}\right)^p.
\end{split}
\end{equation*}
Hence the result (a) follows. The case $\beta = 0$ follows from the calculation in (b).
\item[(b)] If $\beta > 1/2$, then
\[
\sum_{\ell = 1}^{L} \frac{2^{p\ell(1/2 - \beta)}}{\lambda_{\ell}^p} = \sum_{\ell = 1}^L \frac{1}{\ell^p2^{p\ell(\beta - 1/2)}},
\]
and so, the bound for this case follows from (a).

If $\beta = 1/2$, then
\begin{equation}\label{eq:BoundXi2Equal}
\sum_{\ell = 1}^{L} \frac{2^{p\ell(1/2 - \beta)}}{\lambda_{\ell}^p} = \sum_{\ell = 1}^L \frac{1}{\lambda_{\ell}^p} \le 2\left(\frac{\pi^2}{3}\right)^p\sum_{\ell = 1}^{L/2} \ell^{2p} \le 2\left(\frac{\pi^2}{3}\right)^p\left(\frac{\log n}{\log 2}\right)^{2p+1}.
\end{equation}

If $\beta > 1/2$, then
\begin{equation}\label{eq:BoundXi2Smaller}
\begin{split}
\sum_{\ell = 1}^{L} &\frac{2^{p\ell(1/2 - \beta)}}{\lambda_{\ell}^p}\\ &= \sum_{\ell = 1}^{L/2}\frac{2^{\ell(1/2 - \beta)p}}{\lambda_{\ell}^p} + 2^{(L + 1)(1/2 - \beta)p}\sum_{\ell = 1}^{L/2} \frac{1}{\lambda_{\ell}^{p}2^{\ell(1/2 - \beta)p}}\\
&\le \sum_{\ell = 1}^{L/2}\frac{2^{\ell(1/2 - \beta)p}}{\lambda_{\ell}^p} + (2n)^{(1/2 - \beta)p}\sum_{\ell = 1}^{L/2} \frac{1}{\lambda_{\ell}^{p}2^{\ell(1/2 - \beta)p}}\\
&\le 2^{(L+1)(1/2 - \beta)p}\sum_{\ell = 1}^{L/2}\frac{1}{\lambda_{\ell}^p2^{(L + 1 - \ell)(1/2 - \beta)p}} + (2n)^{(1/2 - \beta)p}\sum_{\ell = 1}^{L/2} \frac{1}{\lambda_{\ell}^{p}2^{\ell(1/2 - \beta)p}}\\
&\le (2n)^{(1/2 - \beta)p}\sum_{\ell = 1}^{L/2}\frac{1}{\lambda_{\ell}^p2^{\ell(1/2 - \beta)p}} + (2n)^{(1/2 - \beta)p}\sum_{\ell = 1}^{L/2} \frac{1}{\lambda_{\ell}^{p}2^{\ell(1/2 - \beta)p}}\\
&\le (2n)^{(1/2 - \beta)p}\left(\frac{5}{(1/2 - \beta)^3}\right)^p\left(\frac{\pi^2}{3}\right)^{p+1}.
\end{split}
\end{equation}
\end{enumerate}
Hence the result follows.
\end{proof}
Define the functions 
\begin{equation}\label{eq:ShiftAndTrunc}
s(\lambda) := (1/2 + 1/\lambda)^{-1},\quad\mbox{and}\quad T_1(\lambda) := \min\{\lambda, 1\}\quad\mbox{for all}\quad \lambda > 0. 
\end{equation}
\begin{thm}\label{thm:DependentSumsApp}
Suppose $Z_1, \ldots, Z_n$ are elements of the causal process \eqref{eq:CausalProcess} with mean zero. If for some $\alpha > 0$, and $\nu > 0$,
\begin{equation}\label{eq:AssumptionDependence}
\norm{\{Z\}}_{\psi_{\alpha}, \nu} = \sup_{p\ge 2}\,\sup_{m\ge 0}\, p^{-1/\alpha}(m+ 1)^{\nu}\Delta_{m,p}< \infty.
\end{equation}
Define
\[
\Omega_n(\nu) := 2^{\nu}\times\begin{cases}
5/(\nu - 1/2)^3, &\mbox{if }\nu > 1/2,\\
2(\log_2n)^{5/2}, &\mbox{if }\nu = 1/2,\\
5(2n)^{(1/2 - \nu)}/(1/2 - \nu)^3,&\mbox{if }\nu < 1/2.
\end{cases}
\]
Then for any $p\ge 2$, 
\begin{equation}\label{eq:MomentBound}
\norm{\sum_{i=1}^n Z_i}_p \le \sqrt{pn}\norm{\{Z\}}_{\psi_{\alpha},\nu}B_{\nu}
 + C_{\alpha}\norm{\{Z\}}_{\psi_{\alpha}, \nu}(\log n)^{1/s(\alpha)}p^{1/T_1(s(\alpha))}\Omega_n(\nu),
\end{equation}
where $C_{\alpha}$ is a constant depending only on $\alpha$, $B_{\nu}$ is a constant depending only on $\nu$ given by
\[
B_{\nu} := \sqrt{6}\left[1 + \frac{20\pi^32^{\nu}}{3\sqrt{3}\nu^3}\right],\quad\mbox{if}\quad \nu > 0.
\]
Furthermore, it follows by Markov's inequality that for all $t\ge 0$,
\begin{equation}\label{eq:TailBound}
\mathbb{P}\left(\left|\sum_{i=1}^n Z_i\right| \ge e\sqrt{tn}\norm{\{Z\}}_{2,\nu}B_{\nu} + C_{\alpha}\norm{\{Z\}}_{\psi_{\alpha},\nu}t^{1/T_1(s(\alpha))}(\log n)^{1/s(\alpha)}\Omega_n(\nu)\right) \le 8e^{-t}.
\end{equation}
Here $C_{\alpha}$ is different from the one in the moment bound~\eqref{eq:MomentBound}.
\end{thm}
\begin{proof}
Define
\[
S_n := \sum_{i=1}^n Z_i,\quad L = \left\lfloor\frac{\log n}{\log 2}\right\rfloor,\quad\mbox{and}\quad \xi_{\ell} =\begin{cases} 2^{\ell},&\mbox{if }0\le \ell < L,\\n,&\mbox{if }\ell = L.\end{cases}
\]
Define for $m\ge 0$,
\begin{equation}\label{eq:DefinitionZM}
Z_i^{(m)} := \mathbb{E}\left[Z_i\big|\varepsilon_{i - m}, \ldots, \varepsilon_i\right],\quad\mbox{and}\quad M_{i,\ell} := \sum_{k = 1}^i \left(Z_{k}^{(\xi_{\ell})} - Z_{k}^{(\xi_{\ell - 1})}\right).
\end{equation}
Let
\[
S_{n,m} := \sum_{i = 1}^n Z_i^{(m)},
\]
and consider the decomposition
\begin{equation}\label{eq:Decomposition}
S_n = S_{n,0} + \left(S_n - S_{n,n}\right) + \sum_{\ell = 1}^L \left(S_{n,\xi_{\ell}} - S_{n,\xi_{\ell - 1}}\right) := \mathbf{I} + \mathbf{II} + \mathbf{III}.
\end{equation}
We prove the moment bound~\eqref{eq:MomentBound} by bounding the moments of each term in the decomposition~\eqref{eq:Decomposition}. 

\emph{Bounding $\mathbf{I}$}: Regarding the first term $\mathbf{I}$, observe that $S_{n,0}$ is a sum of independent random variables $Z_i^{(0)}$ satisfying the tail assumption of Theorem~\ref{thm:IndependenceTailBound} with $\beta = \alpha$. This verification follows by noting that
\[
\norm{Z_i^{(0)}}_p \overset{(a)}{\le} \norm{Z_i}_p \overset{(b)}{\le} \norm{\{Z\}}_{p,\nu} \overset{(c)}{\le} p^{1/\alpha}\norm{\{Z\}}_{\psi_{\alpha},\nu}.
\] 
Inequality (a) follows from Jensen's inequality, (b) follows from Proposition~\ref{prop:Converse} and (c) follows from assumption~\eqref{eq:AssumptionDependence}. Hence, we get that for any $p\ge 1$,
\begin{equation*}
\begin{split}
\norm{\mathbf{I}}_p &= \norm{\sum_{i=1}^n \mathbb{E}\left[Z_i\big|\varepsilon_i\right]}_p\\
&\le \sqrt{6p}\left(\sum_{i=1}^n \mathbb{E}\left[Z_i^2\right]\right)^{1/2} + C_{\alpha}\norm{\{Z\}}_{\psi_{\alpha},\nu}p^{1/T_1(\alpha)}\left(\log n\right)^{1/\alpha},
\end{split}
\end{equation*}
for some constant $C_{\alpha}$ depending only on $\alpha$. Here Jensen's inequality is used to bound the variance of $\mathbb{E}\left[Z_i\big|\varepsilon_i\right]$. By Proposition~\ref{prop:Converse}, $\norm{Z_i}_2 \le \norm{\{Z\}}_{2,\nu}$ and hence
\begin{equation}\label{eq:FirstTermBound}
\norm{S_{n,0}}_p \le \sqrt{6pn}\norm{\{Z\}}_{2,\nu} + C_{\alpha}\norm{\{Z\}}_{\psi_{\alpha},\nu}p^{1/T_1(\alpha)}\left(\log n\right)^{1/\alpha}.
\end{equation}

\emph{Bounding $\mathbf{II}$}: For the second term, note that
\[
S_n = \sum_{i=1}^n Z_i = \sum_{i=1}^n \mathbb{E}\left[Z_i\big|\varepsilon_i, \varepsilon_{i-1}, \ldots\right] = S_{n,\infty},
\]
and hence,
\[
S_n - S_{n,n} = \sum_{m = n}^{\infty}\left(S_{n,m+1} - S_{n,m}\right).
\]
Substituting the definition of $S_{n,m}$, we have
\[
S_{n,m+1} - S_{n,m} = \sum_{k = 1}^n \left(\mathbb{E}\left[Z_k\big|\varepsilon_k, \ldots, \varepsilon_{k-m-1}\right] - \mathbb{E}\left[Z_k\big|\varepsilon_k, \ldots, \varepsilon_{k-m}\right]\right).
\]
We now prove that the summands above form a martingale difference sequence with respect to a filtration. The following construction is taken from the proof of Lemma 1 of \cite{Liu10}. Define
\[
D_{k,m+1} := \mathbb{E}\left[Z_k\big|\varepsilon_k, \ldots, \varepsilon_{k-m-1}\right] - \mathbb{E}\left[Z_k\big|\varepsilon_k, \ldots, \varepsilon_{k-m}\right],
\]
and the non-decreasing filtration
\[
\mathcal{G}_{k, m+1} := \sigma\left(\varepsilon_{k-m-1}, \varepsilon_{k-m-1}, \ldots\right). 
\]
It is easy to see that
\begin{equation}\label{eq:MDSProof}
\mathbb{E}\left[D_{n-k+1, m+1}\big|\mathcal{G}_{k-1, m+1}\right] = 0.
\end{equation}
Therefore, $\{(D_{n-k+1,m+1}, \mathcal{G}_{k,m+1}):\, 1\le k\le n\}$ forms a martingale difference sequence. This implies that $S_{n,m+1} - S_{n,m}$ is a martingale and hence by Lemma \ref{lem:MartingaleBounds} we get for $p\ge 2$,
\[
\norm{S_{n,m+1} - S_{n,m}}_p^2 \le p\sum_{k = 1}^n \norm{D_{k,m+1}}_p^2.
\]
To further bound the right hand side, note that for $p\ge 2$,
\begin{equation}\label{eq:MDSBounds}
\norm{D_{k,m+1}}_p = \norm{\mathbb{E}\left[Z_k - g(\ldots, \varepsilon_{k-m-1}', \varepsilon_{k-m}, \ldots, \varepsilon_k)\big|\varepsilon_k, \ldots, \varepsilon_{k-m-1}\right]}_p \le \delta_{m+1, p}.
\end{equation}
Hence, for $p\ge 2$,
\[
\norm{S_{n,m+1} - S_{n,m}}_p \le \sqrt{pn}\delta_{m+1,p},
\]
and
\begin{equation}\label{eq:StopPolynomial2}
\norm{S_n - S_{n,n}}_p \le \sum_{m = n}^{\infty} \norm{S_{n,m+1} - S_{n,m}}_p \le \sqrt{pn}\sum_{m = n}^{\infty} \delta_{m+1, p} = \sqrt{pn}\Delta_{n+1, p}.
\end{equation}
Under assumption \eqref{eq:AssumptionDependence}, we obtain
\begin{equation}\label{eq:SecondTermBound}
\norm{\mathbf{II}}_p = \norm{S_n - S_{n,n}}_p \le \norm{\{Z\}}_{\psi_{\alpha},\nu}\frac{{n}^{1/2}p^{1/2 + 1/\alpha}}{(n + 2)^{\nu}} = \norm{\{Z\}}_{\psi_{\alpha},\nu}n^{1/2 - \nu}p^{1/2 + 1/\alpha}.
\end{equation}

\emph{Bounding $\mathbf{III}$}: To bound $\mathbf{III}$, note by definition of $M_{i,\ell}$ that
\[
\mathbf{III} = \sum_{\ell = 1}^L \sum_{k = 1}^n \left(Z_k^{(\xi_{\ell})} - Z_k^{(\xi_{\ell - 1})}\right) = \sum_{\ell = 1}^L M_{n,\ell}.
\]
Now observe that the summands of $M_{n,\ell}$,
$$\mathcal{D}_{k, \ell} := \left(Z_{k}^{(\xi_{\ell})} - Z_k^{(\xi_{\ell - 1})}\right),$$
are $\xi_{\ell}$-dependent in the sense that $\mathcal{D}_{k,\ell}$ and $\mathcal{D}_{s, \ell}$ are independent if $|s - k| > \xi_{\ell}$. This can be proved as follows. By definition $\mathcal{D}_{k,\ell}$ is only a function of $(\varepsilon_k, \ldots, \varepsilon_{k - \xi_{\ell}})$ and by independence of $\varepsilon_k, k\in\mathbb{Z}$, the claim follows. Now a blocking technique can be used to convert $M_{n,\ell}$ into a sum of independent variables. See Corollary A.1 of \cite{RomanoWolf00} for a similar use. Define
\begin{align*}
\mathcal{A}_{\ell} &:= \left\{2\xi_{\ell}i + j:\,i\in\mathbb{Z},\,1\le j\le \xi_{\ell}\right\},\\
\mathcal{B}_{\ell} &:= \left\{2\xi_{\ell}i + \xi_{\ell} + j:\,i\in\mathbb{Z},\,1\le j\le \xi_{\ell}\right\}.
\end{align*}
Consider the decomposition of $M_{n,\ell}$ as
\[
M_{n,\ell} = \sum_{k = 1}^n \mathcal{D}_{k,\ell} = A_{n,\ell} + B_{n,\ell},
\]
where
\begin{align*}
A_{n,\ell} := \sum_{1\le k\le n, k\in\mathcal{A}} \mathcal{D}_{k,\ell} \quad\mbox{and}\quad B_{n,\ell} := \sum_{1\le k\le n, k\in\mathcal{B}} \mathcal{D}_{k,\ell}.
\end{align*}
We now provide moment bounds for $M_{n,\ell}$ by giving moment bounds for $A_{n,\ell}$ and $B_{n,\ell}$ which is in turn done by separating the summands of $A_{n,\ell}$ and $B_{n,\ell}$ to form an independent sum. Note that
\begin{equation}\label{eq:AnellDecomposition}
\begin{split}
A_{n,\ell} &= \sum_{i = 1}^{\left\lfloor \frac{n}{2\xi_{\ell}}\right\rfloor}\left(\sum_{j = 1}^{\xi_{\ell}} \mathcal{D}_{2\xi_{\ell}i + j,\ell}\right) = \sum_{i = 1}^{\left\lfloor \frac{n}{2\xi_{\ell}}\right\rfloor} \left(\sum_{k = 2\xi_{\ell}i + 1}^{2\xi_{\ell}i + \xi_{\ell}} \left(Z_k^{(\xi_{\ell})} - Z_k^{(\xi_{\ell - 1})}\right)\right)\\
&= \sum_{i = 1}^{\left\lfloor \frac{n}{2\xi_{\ell}}\right\rfloor} \left(M_{2\xi_{\ell}i + \xi_{\ell}, \ell} - M_{2\xi_{\ell}i,\ell}\right).
\end{split}
\end{equation}
By the $\xi_{\ell}$-independence of the summands of $M_{n,\ell}$, we get that the summands in the final representation of $A_{n,\ell}$ are independent and so Theorem~\ref{thm:IndependenceTailBound} applies. In the following, we verify the assumption of Theorem~\ref{thm:IndependenceTailBound}. 
For $1\le i < j\le n$, it is clear that
\begin{align*}
M_{j,\ell} - M_{i,\ell} &= \sum_{k = i+1}^j \left(Z_k^{(\xi_{\ell})} - Z_k^{(\xi_{\ell - 1})}\right)\\
&= \sum_{k = i+1}^j \left(\sum_{t = 1 + \xi_{\ell - 1}}^{\xi_{\ell}} \left(Z_k^{\xi_{\ell}} - Z_k^{(\xi_{\ell - 1})}\right)\right)\\
&= \sum_{t = 1 + \xi_{\ell - 1}}^{\xi_{\ell}}\left(\sum_{k = i + 1}^{j} \left(Z_k^{(t)} - Z_k^{(t - 1)}\right)\right).
\end{align*}
By triangle inequality
\begin{equation}\label{eq:DifferenceMnBounds}
\norm{M_{j,\ell} - M_{i,\ell}}_p \le \sum_{t = 1 + \xi_{\ell - 1}}^{\xi_{\ell}} \norm{\sum_{k = i+1}^j \left(Z_k^{(t)} - Z_k^{(t - 1)}\right)}_p.
\end{equation}
As proved in \eqref{eq:MDSProof}, the summation for each $t$ represents a martingale and hence by Lemma \ref{lem:MartingaleBounds}, we get for $p \ge 2$ that
\[
\norm{\sum_{k = i+1}^j \left(Z_k^{(t)} - Z_k^{(t - 1)}\right)}_p^2 \le p\sum_{k = i+1}^j \norm{Z_k^{(t)} - Z_k^{(t - 1)}}_p^2 \le p\sum_{k = i + 1}^j \delta_{t,p}^2 = p(j - i)\delta_{t,p}^2.
\]
Here we used inequality \eqref{eq:MDSBounds}. Substituting this in inequality \eqref{eq:DifferenceMnBounds} and using $\xi_{\ell - 1} \ge \xi_{\ell}/2$, we get
\begin{equation}\label{eq:pMomentofDiffM}
\begin{split}
\norm{M_{j,\ell} - M_{i,\ell}}_p &\le p^{1/2}(j-i)^{1/2}\sum_{t = 1 + \xi_{\ell - 1}}^{\xi_{\ell}} \delta_{t, p} \le p^{1/2}(j-i)^{1/2}\Delta_{1 + \xi_{\ell - 1}, p}\\
&\le \norm{\{Z\}}_{p,\nu}p^{1/2}(j- i)^{1/2}(2 + \xi_{\ell - 1})^{-\nu}\\ &\le 2^{\nu}\norm{\{Z\}}_{p,\nu}p^{1/2}(j- i)^{1/2}\xi_{\ell}^{-\nu}.
\end{split}
\end{equation}
Under assumption \eqref{eq:AssumptionDependence}, we get
\begin{equation}\label{eq:SubstituteAssumptionpMomentofDiffM}
\begin{split}
\norm{M_{j,\ell} - M_{i,\ell}}_p &\le 2^{\nu}\norm{\{Z\}}_{\psi_{\alpha}, \nu}p^{1/2 + 1/\alpha}(j-i)^{1/2}\xi_{\ell}^{-\nu}\\
&= 2^{\nu}\norm{\{Z\}}_{\psi_{\alpha}, \nu}p^{1/s(\alpha)}(j-i)^{1/2}\xi_{\ell}^{-\nu}.
\end{split}
\end{equation}
See~\eqref{eq:ShiftAndTrunc} for the definition of $s(\alpha)$. Thus, for all $1\le i\le \lfloor \frac{n}{2\xi_{\ell}}\rfloor$,
\begin{equation}\label{eq:VerificationTailAssumptionIndependence}
\sup_{p\ge 2}\,p^{- 1/s(\alpha)}\norm{M_{2\xi_{\ell}i+\xi_{\ell},\ell} - M_{2\xi_{\ell}i,\ell}}_p \le 2^{\nu}\norm{\{Z\}}_{\psi_{\alpha},\nu}{\xi_{\ell}^{1/2 - \nu}}.
\end{equation}
So, the summands of $A_{n,\ell}$ in the final representation in \eqref{eq:AnellDecomposition} are independent and satisfy the hypothesis of Theorem~\ref{thm:IndependenceTailBound} with $\beta = s(\alpha)$. Therefore, for $p\ge 2$,
\begin{align*}
\norm{A_{n,\ell}}_p &\le \sqrt{6p}\left(\sum_{i = 1}^{\lfloor n/(2\xi_{\ell})\rfloor} \norm{M_{2\xi_{\ell}i + \xi_{\ell}, \ell} - M_{2\xi_{\ell}i, \ell}}_2^2\right)^{1/2}\\ &\qquad+ C_{\alpha}2^{\nu}\norm{\{Z\}}_{\psi_{\alpha}, \nu}\left(\log n\right)^{1/s(\alpha)}{\xi_{\ell}^{1/2 - \nu}}p^{1/T_1(s(\alpha))}\\
&\le \sqrt{12p}\norm{\{Z\}}_{2, \nu}\frac{2^{\nu}\xi_{\ell}^{1/2}}{\xi_{\ell}^{\nu}}\left(\frac{n}{2\xi_{\ell}}\right)^{1/2}\\
&\qquad+ C_{\alpha}2^{\nu}\norm{\{Z\}}_{\psi_{\alpha}, \nu}\left(\log n\right)^{1/s(\alpha)}{\xi_{\ell}^{1/2 - \nu}}p^{1/T_1(s(\alpha))}\\
&\le \frac{2^{\nu}}{\xi_{\ell}^{\nu}}\left[\norm{\{Z\}}_{2,\nu}\sqrt{6pn} + C_{\alpha}\norm{\{Z\}}_{\psi_{\alpha},\nu}p^{1/T_1(s(\alpha))}\left(\log n\right)^{1/s(\alpha)}\xi_{\ell}^{1/2}\right].
\end{align*}
Here the second inequality follows from \eqref{eq:pMomentofDiffM}.

Similarly a representation for $B_{n,\ell}$ exists with independent summands satisfying the assumption of Theorem~\ref{thm:IndependenceTailBound} with $\beta = s(\alpha)$ and so,
\[
\norm{B_{n,\ell}}_p \le \frac{2^{\nu}}{\xi_{\ell}^{\nu}}\left[\norm{\{Z\}}_{2,\nu}\sqrt{6pn} + C_{\alpha}\norm{\{Z\}}_{\psi_{\alpha},\nu}p^{1/T_1(s(\alpha))}\left(\log n\right)^{1/s(\alpha)}\xi_{\ell}^{1/2}\right].
\]
Combining the bounds for $A_{n,\ell}$ and $B_{n,\ell}$ implies the bound on $M_{n,\ell}$ as
\begin{equation}\label{eq:MnellBound}
\norm{M_{n,\ell}}_p \le \frac{2^{1 + \nu}}{\xi_{\ell}^{\nu}}\left[\norm{\{Z\}}_{2,\nu}\sqrt{6pn} + C_{\alpha}\norm{\{Z\}}_{\psi_{\alpha},\nu}p^{1/T_1(s(\alpha))}\left(\log n\right)^{1/s(\alpha)}\xi_{\ell}^{1/2}\right].
\end{equation}
To complete bounding $\mathbf{III}$, we need to bound the moments of the sum of $M_{n,\ell}$ over $1\le \ell \le L$ which are all dependent. For this, define the sequence
\[
\lambda_{\ell} = \begin{cases}
3\pi^{-2}\ell^{-2},&\mbox{if }1\le \ell \le L/2,\\
3\pi^{-2}(L + 1 - \ell)^{-2},&\mbox{if }L/2 < \ell \le L.
\end{cases}
\]
This positive sequence satisfies $\sum_{\ell = 1}^L \lambda_{\ell} < 1$. It is easy to derive from H\''older's inequality that
\begin{equation}\label{eq:HolderImplication}
\left|\sum_{\ell = 1}^L a_{\ell}\right|^p \le \sum_{\ell = 1}^L \frac{|a_{\ell}|^p}{\lambda_{\ell}^p}.
\end{equation}
Substituting in this inequality $a_{\ell} = M_{n,\ell}$ and the moment bound \eqref{eq:MnellBound}, we get
\begin{align*}
\mathbb{E}\left[\left|\sum_{\ell = 1}^L M_{n,\ell}\right|^p\right] &\le 2^{(2 + \nu)p}\norm{\{Z\}}_{2,\nu}^p(6pn)^{p/2}\sum_{\ell = 1}^L \frac{1}{\lambda_{\ell}^p\xi_{\ell}^{p\nu}}\\ &\quad+ C_{\alpha}^p2^{(2 + \nu)p}\norm{\{Z\}}_{\psi_{\alpha},\nu}^pp^{p/T_1(s(\alpha))}\left(\log n\right)^{p/s(\alpha)}\sum_{\ell = 1}^L \frac{\xi_{\ell}^{p/2}}{\lambda_{\ell}^p\xi_{\ell}^{p\nu}}.
\end{align*}
It follows from Lemma~\ref{lem:SimpleCalculation} and the definition of $\Omega_n(\nu)$ that for $p\ge 2$,
\begin{equation}\label{eq:ThirdTermBound}
\begin{split}
&\norm{\sum_{\ell = 1}^L M_{n,\ell}}_p\\ &\qquad\le \frac{5\pi^32^{2}}{3\sqrt{3}}\left[\frac{2^{\nu}\norm{\{Z\}}_{2,\nu}\sqrt{6pn}}{\nu^3} + C_{\alpha}\norm{\{Z\}}_{\psi_{\alpha}, \nu}(\log n)^{1/s(\alpha)}\Omega_n(\nu)p^{1/T_1(s(\alpha))}\right].
\end{split}
\end{equation}
Combining the moment bounds \eqref{eq:FirstTermBound}, \eqref{eq:SecondTermBound} and \eqref{eq:ThirdTermBound}, it follows that for $p\ge 2$,
\begin{align*}
\norm{S_n}_p &\le \sqrt{6pn}\norm{\{Z\}}_{\psi_{\alpha},\nu}\left[1 + \frac{20\pi^32^{\nu}}{3\sqrt{3}\nu^3}\right] + \norm{\{Z\}}_{\psi_{\alpha}, \nu}n^{1/2 - \nu}p^{1/s(\alpha)}\\
&\quad+ C_{\alpha}\norm{\{Z\}}_{\psi_{\alpha}, \nu}(\log n)^{1/s(\alpha)}p^{1/T_1(s(\alpha))}\Omega_n(\nu).
\end{align*}
Here the inequalities $s(\alpha) \le \alpha$ and $T_1(s(\alpha)) \le T_1(\alpha)$ are used. Now noting that $\Omega_n(\nu) \ge n^{1/2 - \nu}$ for all $\nu > 0$ and $p^{1/s(\alpha)} \le p^{1/T_1(s(\alpha))}$, the result follows.
\end{proof}
In the following two lemmas, we prove that the dependent adjusted norm of linear combinations and products of functionally dependent random variables can be bounded in terms of the individual processes. Recall the definition of $\Theta_k$ from~\eqref{eq:kSparseSubsetDef}.
\begin{lem}\label{lem:LinearCombinations}
Suppose Assumption~\ref{eq:Dependent} holds, then for any $\theta\in\Theta_k$,
\[
\sup_{\theta\in\Theta_k}\norm{\{\theta^{\top}X\}}_{r,\nu} \le k^{1/2}K_{n,p}.
\]
\end{lem}
\begin{proof}
Fix $\theta\in\Theta_k$. Set the functional dependence measure~\eqref{eq:FunctionalDepMeasure} for the linear combination $\theta^{\top}X$ as
\[
\delta_{s, r}^{(L)} := \max_{1\le i\le n}\norm{\theta^{\top}X_i - \theta^{\top}X_{i,s}}_r.
\]
Note that $\theta\in\Theta_k$ are all $k$-sparse and so there are only $k$ non-zero coordinates $\theta(j)$ of $\theta$. Since the functional dependence measure is a norm, it follows that
\begin{align*}
\delta_{s, r}^{(L)} &= \max_{1\le i\le n}\sum_{j = 1}^p |\theta(j)|\norm{X_i(j) - X_{i,s}(j)}_r\\
&\le \sum_{j = 1}^p |\theta(j)|\max_{1\le i\le n}\norm{X_i(j) - X_{i,s}(j)}_r = \sum_{j = 1}^p |\theta(j)|\delta_{s, r, j}.
\end{align*}
Hence for $m\ge 0,$
\[
\Delta_{m, r}^{(L)} := \sum_{s = m}^{\infty} \delta_{s, r}^{(L)} \le \sum_{s = m}^{\infty} \sum_{j = 1}^p |\theta(j)|\delta_{s, r, j} = \sum_{j = 1}^p |\theta(j)|\left(\sum_{s = m}^{\infty} \delta_{s, r, j}\right) = \sum_{j = 1}^p |\theta(j)|\Delta_{m, r, j}.
\]
This implies that
\[
\Delta_{m,r}^{(L)} \le \norm{\theta}_1\max_{1\le j\le p}\Delta_{m, r, j} \le k^{1/2}\max_{1\le j\le p}\Delta_{m, r ,j}.
\]
Therefore, for $r\ge 1$ and $\nu > 0$,
\[
\norm{\{\theta^{\top}X\}}_{r, \nu} \le k^{1/2}\norm{\{X\}}_{r, \nu}\quad\Rightarrow\quad \norm{\{\theta^{\top}X\}}_{\psi_{\alpha}, \nu} \le k^{1/2}\norm{\{X\}}_{\psi_{\alpha}, \nu} \le k^{1/2}K_{n,p},
\]
proving the result.
\end{proof}
\begin{lem}\label{lem:ProductProcess}
Suppose $(W_1^{(1)}, W_1^{(2)}), \ldots, (W_n^{(1)}, W_n^{(2)})$ are $n$ functionally dependent real-valued random vectors. Set $W_i = W_i^{(1)}W_i^{(2)}$ for $1\le i\le n$. Then for all $r\ge 2$ and $\nu > 0$
\begin{align*}
\norm{\{W\}}_{r/2,\nu} &\le \norm{\{W^{(1)}\}}_{r,0}\norm{\{W^{(2)}\}}_{r,\nu} + \max_{1\le i\le n}\left|\mathbb{E}\left[W_i^{(1)}\right]\right|\norm{\{W^{(2)}\}}_{r,\nu}\\ &\quad+ \norm{\{W^{(2)}\}}_{r,0}\norm{\{W^{(1)}\}}_{r,\nu} + \max_{1\le i\le n}\left|\mathbb{E}\left[W_i^{(2)}\right]\right|\norm{\{W^{(1)}\}}_{r,\nu}.
\end{align*}
\end{lem}
\begin{proof}
Set for $j = 1,2,$
\[
\delta_{s,r}^{(j)} := \norm{W_i^{(1)} - W_{i,s}^{(1)}}_r,\quad\mbox{and}\quad \Delta_{m,r}^{(j)} := \sum_{s = m}^{\infty} \delta_{s,r}^{(j)}.
\]
Fix $1\le i\le n$ and consider 
\begin{align*}
\varphi_{s,r/2,i} &:= \norm{W_i^{(1)}W_i^{(2)} - W_{i,s}^{(1)}W_{i,s}^{(2)}}_{r/2}\\
&= \norm{W_{i}^{(1)}\left[W_i^{(2)} - W_{i,s}^{(2)}\right] + W_{i,s}^{(2)}\left[W_{i}^{(1)} - W_{i,s}^{(1)}\right]}_{r/2}\\
&\le \norm{W_{i}^{(1)}\left[W_i^{(2)} - W_{i,s}^{(2)}\right]}_{r/2} + \norm{W_{i,s}^{(2)}\left[W_{i}^{(1)} - W_{i,s}^{(1)}\right]}_{r/2}\\
&\le \norm{W_i^{(1)}}_r\norm{W_i^{(2)} - W_{i,s}^{(2)}}_{r} + \norm{W_{i,s}^{(2)}}_r\norm{W_i^{(1)} - W_{i,s}^{(1)}}_r\\
&\le \norm{W_i^{(1)}}_r\delta_{k,r}^{(2)} + \norm{W_{i,s}^{(2)}}_r\delta_{k,r}^{(1)}.
\end{align*}
Since $\varepsilon_{i-k}'$ is identically distributed as $\varepsilon_{i-k}$, $\norm{W_{i,s}^{(2)}}_r = \norm{W_{i}^{(2)}}_r$. So, an upper bound on the dependence adjusted norm can be obtained as
\begin{align*}
\Delta_{m,r/2} = \sum_{k=m}^{\infty} \max_{1\le i\le n} \,\varphi_{k,r/2,i} &\le \max_{1\le i\le n}\norm{W_i^{(1)}}_r\sum_{k=m}^{\infty} \delta_{k,r}^{(2)} + \max_{1\le i\le n}\norm{W_{i}^{(2)}}_r\sum_{k = m}^{\infty} \delta_{k,r}^{(1)}\\
&\le \max_{1\le i\le n} \norm{W_i^{(1)}}_r\Delta_{m,r}^{(2)} + \max_{1\le i\le n}\norm{W_i^{(2)}}_{r}\Delta_{m,r}^{(1)},
\end{align*}
and thus,
\begin{align*}
\norm{\{W\}}_{r/2,\nu} &\le \max_{1\le i\le n} \norm{W_i^{(1)}}_r\norm{\{W^{(2)}\}}_{r,\nu} + \max_{1\le i\le n}\norm{W_i^{(2)}}_{r}\norm{\{W^{(1)}\}}_{r,\nu}\\
&\le \norm{\{W^{(1)}\}}_{r,0}\norm{\{W^{(2)}\}}_{r,\nu} + \max_{1\le i\le n}\left|\mathbb{E}\left[W_i^{(1)}\right]\right|\norm{\{W^{(2)}\}}_{r,\nu}\\ &\quad+ \norm{\{W^{(2)}\}}_{r,0}\norm{\{W^{(1)}\}}_{r,\nu} + \max_{1\le i\le n}\left|\mathbb{E}\left[W_i^{(2)}\right]\right|\norm{\{W^{(1)}\}}_{r,\nu},
\end{align*}
proving the result.
\end{proof}
\section{Proof of Proposition~\ref{prop:UniformNonSing}}\label{appsec:propUniformNonSing}
\begin{proof}
It is easy to see that
\begin{align*}
\RIP(k, \Sigma_1 - \Sigma_2) &= \sup_{\substack{\theta\in\mathbb{R}^p, \norm{\theta}_0 \le k,\\\norm{\theta}_2 \le 1}}\left|\theta^{\top}\left(\Sigma_1 - \Sigma_2\right)\theta\right|\\
&\le \sup_{\substack{\theta\in\mathbb{R}^p,\\\norm{\theta}_0 \le k, \norm{\theta}_2 \le 1}}\norm{\theta}_1^2\vertiii{{\Sigma}_1 - \Sigma_2}_{\infty} \le k\vertiii{{\Sigma}_1 - \Sigma_2}_{\infty}.
\end{align*}
Here we have used inequalities \eqref{eq:MatrixVectorIneq}. A similar proof implies the second result.
\end{proof}
\bibliographystyle{apalike}
\bibliography{AssumpLean}
\end{document}